\documentclass[12pt]{amsart}

\usepackage[T1]{fontenc}
\usepackage{amsmath,amsthm,amscd}
\usepackage{amssymb}

\usepackage{newtxtext}
\usepackage{newtxmath}

\usepackage{microtype}
\usepackage{graphicx}
\usepackage{xcolor} 
\usepackage{bm,cancel,enumitem,chessboard}
\usepackage{tikz}
\usepackage[margin=1in]{geometry}

\usepackage[all,cmtip]{xy}
\usepackage{hyperref} 

\numberwithin{equation}{section}

\newtheorem{theorem}{Theorem}[section]
\newtheorem{proposition}[theorem]{Proposition}
\newtheorem{corollary}[theorem]{Corollary}
\newtheorem{lemma}[theorem]{Lemma}

\newtheorem{problem}[theorem]{Problem}

\newtheorem{remark}[theorem]{Remark}
\newtheorem{definition}[theorem]{Definition}

\theoremstyle{definition}
\newtheorem{defn}[theorem]{Definition}

\newcommand{\Hilb}{{\mathrm{Hilb}}}

\newcommand{\symm}{{\mathfrak{S}}}

\newcommand{\Sym}{{\mathrm{Sym}}}

\newcommand{\zero}{{\mathbf{0}}}

\newcommand{\sign}{{\mathrm{sign}}}

\newcommand{\DDD}{{\mathfrak{D}}}

\newcommand{\AAA}{{\mathcal{A}}}
\newcommand{\BBB}{{\mathcal{B}}}

\newcommand{\FF}{{\mathbb{F}}}
\newcommand{\CC}{{\mathbb{C}}}

\newcommand{\ZZ}{{\mathbb{Z}}}

\newcommand{\DD}{{\mathbb{D}}}

\newcommand{\xx}{{\mathbf{x}}}

\newcommand{\zz}{{\mathbf{z}}}

\newcommand{\ann}{{\mathrm{ann}}}

\newcommand{\Gale}{{\mathrm{Gale}}}

\begin{document}

\title[Superspace coinvariants and inverse systems for $GL_n(\FF_q)$]
{Superspace coinvariants and inverse systems for $GL_n(\FF_q)$}

\author[Brendon Rhoades]{Brendon Rhoades}
\address{University of California, San Diego}
\email{bprhoades@ucsd.edu}

\author[Andy Wilson]{Andy Wilson}
\address{Kennesaw State University}
\email{awils342@kennesaw.edu}

\begin{abstract}
    Let $q$ be a prime power and write $\Omega$ for the bigraded algebra of regular differential forms over $\FF_q^n$. The general linear group $GL_n(\FF_q)$ acts on $\Omega$; write $SI \subseteq \Omega$ for the ideal generated by $GL_n(\FF_q)$-invariants with vanishing constant term. The {\em $GL_n(\FF_q)$-superspace coinvariant ring} is the quotient $SR := \Omega/SI$. We calculate the bigraded Hilbert series of $SR$ and give an operator-theoretic characterization of the inverse system $SI^\perp$. Our results extend to subgroups $G$ of $GL_n(\FF_q)$ which contain $SL_n(\FF_q)$. 
\end{abstract}

\maketitle

\section{Introduction}
\label{sec:Introduction}

Let $\FF$ be a field and let $G \subseteq GL_n(\FF)$ be a finite  matrix group over $\FF$. Let $V := \FF^n$ be the defining representation of $G$. The action of $G$ on $V$ induces a graded action on the polynomial ring 
\begin{equation}
S := \Sym(V^*) \cong  \FF[x_1,\dots,x_n]
\end{equation}
of regular functions $V \to \FF$. We write $S^G : = \{ f \in S \,:\, g \cdot f = f \text{ for all } g \in G \}$ for the fixed subalgebra and $I_G := (S^G_+) \subseteq S$ for the ideal generated by $G$-invariants with vanishing constant term. The {\em $G$-coinvariant ring} is 
\begin{equation}
R_G := S/I_G.
\end{equation}
The quotient $R_G$ is a graded $G$-module.

Coinvariant theory is a fundamental topic in algebraic combinatorics whose most popular results occur in the setting where $\FF = \CC$ and $G$ is a complex reflection group. In this context, Chevalley proved \cite{Chevalley} that $R_G$ is isomorphic to the regular representation $\CC[G]$ as an ungraded $G$-module. When $G$ is a Weyl group, Borel proved \cite{Borel} that $R_G$ presents the cohomology of the associated flag variety and Steinberg gave \cite{SteinbergDifferential} a characterization of the inverse system attached to $R_G$ in terms of the $G$-Vandermonde polynomial. 

Let $q$ be a prime power and write $\FF_q$ for the finite field with $q$ elements. There is also a beautiful coinvariant theory associated to the finite general linear group $G = GL_n(\FF_q)$. Dickson gave \cite{Dickson} a list $D_{n,0}, D_{n,1}, \dots, D_{n,n-1} \in S^G$ of polynomial $G$-invariants whose definition will be recalled later. He proved \cite{Dickson} that the $D_{n,i}$ form an algebraically independent generating set of $S^G$. Dickson's original proof was quite complicated; easier proofs were later discovered \cite{SteinbergBasis, Wilkerson}. Mitchell proved \cite[Thm. 1.4]{Mitchell} that the $G$-coinvariant ring $R_G$ is {\em Brauer-isomorphic} to (has the same composition factors as) the regular representation $\FF_q[G]$. Reiner--Stanton--Webb proved \cite{RSW} an enhancement of this fact involving the action of a Singer cycle $c \in G$.

Returning to the setting of an arbitrary field $\FF$, we write $\Omega$ for the bigraded algebra of regular differential forms on $\FF^n$.  That is, we have
\begin{equation}
    \Omega := \Sym(V^*) \otimes \wedge (V^*) \cong \FF[x_1,\dots,x_n] \otimes \wedge \{ \theta_1,\dots,\theta_n\}.
\end{equation}
In physics, the ring $\Omega$ is called the {\em superspace ring} of rank $n$. If $G \subseteq GL_n(\FF)$ is a matrix group, the natural action of $G$ on $V$ induces a bigraded action of $G$ on $\Omega$. We write $\Omega^G \subseteq \Omega$ for the $G$-fixed subalgebra and $SI_G := (\Omega^G_+)\subseteq \Omega$ for the bigraded ideal generated by $G$-invariant elements with vanishing constant term. The {\em $G$-superspace coinvariant ring} is 
\begin{equation}
    SR_G := \Omega/SI_G.
\end{equation}
The quotient $SR_G$ is a bigraded $G$-module with respect to the commutative and exterior gradings.

When $\FF = \CC$, there has been significant activity devoted to extending classical results of coinvariant theory to superspace \cite{ACKMR, Bergeron, DIV, Lentfer, MRW, RWvan, RW, SS, SSComplex, SW, Zabrocki}. The Fields Institute Combinatorics Group\footnote{N. Bergeron, L. Colmenarejo, S. X. Li, J. Machacek, R. Sulzgruber, and M. Zabrocki} made a series of conjectures (see \cite{Zabrocki}) on the structure of $SR_{\symm_n}$. 
These conjectures were proven in a series of papers \cite{ACKMR, MRW, RW} using  hyperplane arrangement theory \cite{AHMMS, AMMN}. 
In forthcoming work \cite{BR}, these results will be extended  to the complex reflection groups $G(r,1,n)$.

Superspace coinvariant theory in the positive characteristic setting $\FF = \FF_q$ has been unexplored until this paper. We focus mainly on the case where $G = GL_n(\FF_q)$ is the full finite general linear group. In Section~\ref{sec:Hilbert} we introduce certain `modular higher Euler operators' $d_j^{(q)}$ which facilitate an understanding of $SR_G$ via its inverse system $SI_G^\perp$. These operators are Frobenius-analogs of the higher Euler operators used in the characteristic 0 setting of \cite{MRW, RW}. The operators $d_j^{(q)}$ are used to prove the following results.
\begin{itemize}
    \item We calculate the bigraded Hilbert series of the $GL_n(\FF_q)$-superspace coinvariant ring (Theorem~\ref{thm:hilbert}). In particular, we give a triangular structure on $SR_G$ indexed by subsets $J \subseteq [n]$ to show that
    \begin{equation}
    \label{eq:intro-hilbert}
        \Hilb(SR_{GL_n(\FF_q)};t,z) = \prod_{i=1}^{n-1} [q^n - q^i]_t \times \sum_{J \subseteq [n]} z^{|J|} \cdot \left[ q^n -1 - \sum_{j \in J} q^{n-j} \right]_t
    \end{equation}
    where $t$ tracks commutative degree and $z$ tracks exterior degree. Recall that the fundamental degrees of $G = GL_n(\FF_q)$ are $q^n-1, q^n-q, \dots, q^n-q^{n-1}$, so that the first factor in the right-hand side of \eqref{eq:intro-hilbert} is the product of $t$-numbers corresponding to the lower $n-1$ degrees. 
    \item We prove (Theorem~\ref{thm:operator}) that the inverse system $SI_G^\perp$ is generated by the operators $d_j^{(q)}$ and (commutative) partial differentiation starting with a certain distinguished element $\delta^{(q)}$ introduced in Proposition~\ref{prop:steinberg}.
\end{itemize}
We prove (Proposition~\ref{prop:steinberg}) a $GL_n(\FF_q)$-analog of Steinberg's characterization \cite{SteinbergDifferential} of the inverse system associated to $R_G$ when $G \subseteq GL_n(\CC)$ is a Weyl group. In particular, we show that the element $\delta^{(q)}$ defined in Proposition~\ref{prop:steinberg} generates the inverse system $I_G^\perp$ under the action of $S$ by partial differentiation. Our results extend nicely to subgroups $G \subseteq GL_n(\FF_q)$ which contain the special linear group $SL_n(\FF_q)$; see Theorems~\ref{thm:hilbert-between} and \ref{thm:operator-between}.

Our proofs make crucial use of two major tools in modular invariant theory.
The first is a result (see Theorem~\ref{thm:HS}) of Hartmann and Shepler \cite{HS} on the structure of the invariant ring $\Omega^G$ when $G$ is a group between $SL_n(\FF_q)$ and $GL_n(\FF_q)$. The structure of this ring is governed by a {\em twisted wedge product} $\curlywedge$ which will play an important role in our proofs. The second tool is a $GL_n(\FF_q)$-analog $s_\lambda^{(q)}$ of the Schur polynomials $s_\lambda$ introduced by Macdonald \cite{MacdonaldVariations}. This modular version of the Schur polynomial is defined using an $\FF_q$-analog of the bialternant formula.

The rest of the paper is organized as follows. In {\bf Section~\ref{sec:Background}} we give background material on commutative algebra, divided power rings, and (super)invariant theory. In {\bf Section~\ref{sec:Upper}} we prove a crucial syzygy (Lemma~\ref{lem:syzygy}) which allows us to bound the Hilbert series of $SR_{GL_n(\FF_q)}$ from above. {\bf Section~\ref{sec:Operator}} proves our operator-theoretic characterization of the commutative inverse system associated to $R_{GL_n(\FF_q)}$. In {\bf Section~\ref{sec:Hilbert}} we prove \eqref{eq:intro-hilbert} and give an operator-theoretic characterization of the inverse system for $SR_{GL_n(\FF_q)}$.  In {\bf Section~\ref{sec:Between}} we describe the adjustments necessary to generalize our work to subgroups $G \subseteq GL_n(\FF_q)$ which contain $SL_n(\FF_q)$. {\bf Section~\ref{sec:Conculsion}} presents two open problems.

\section{Background}
\label{sec:Background}

\subsection{Combinatorics} Given a nonnegative integer $n$, we use the notation
\begin{equation}
    [n]_t := 1 + t + t^2 + \cdots + t^{n-1} = \frac{1-t^n}{1-t}
\end{equation}
for the associated {\em $t$-number}. We also write $[n] := \{1,2,\dots,n\}$. A {\em partition of $n$} is a weakly decreasing sequence $\lambda= (\lambda_1 \geq \cdots \geq \lambda_k)$ of positive integers with $\lambda_ 1+ \cdots + \lambda_k = n$.  Trailing zeros are ignored when considering partitions, so that (for example) $(3,2,2)$ and $(3,2,2,0,0)$ represent the same partition of 7.

Given two positive integers $k \leq n$, the {\em Gale order} is the partial order $\leq_\Gale$ on $k$-element subsets of $[n]$ defined as follows. Suppose $I = \{i_1 < \cdots < i_k \}$ and $J = \{j_1 < \cdots < j_k \}$ are two $k$-element subsets of $[n]$. We have
\begin{equation}
    I \leq_\Gale J \quad \Leftrightarrow \quad i_r \leq j_r \text{ for all } r= 1,2,\dots,k.
\end{equation}

\subsection{Commutative algebra}  Let $\FF$ be an arbitrary field. If $V = \bigoplus_{i \geq 0} V_i$ is a graded $\FF$-vector space with each graded piece $V_i$ finite-dimensional, the {\em Hilbert series} of $V$ is the formal power series
\begin{equation}
    \Hilb(V;t) := \sum_{i \geq 0} \dim_\FF V_i \cdot t^i
\end{equation}
where $t$ is a formal variable. In this paper, we will consider bigraded vector spaces $V = \bigoplus_{i,j \geq 0} V_{i,j}$ where $i$ corresponds to a commutative grading and $j$ corresponds to an exterior grading. In this setting, the {\em bigraded Hilbert series} is
\begin{equation}
    \Hilb(V;t,z) := \sum_{i,j \geq 0} \dim_\FF V_{i,j} \cdot t^i z^j.
\end{equation}

We write $S = \FF[x_1,\dots,x_n]$ for the rank $n$ polynomial ring with its standard grading induced by $\deg(x_i) = 1$ for all $i$. A sequence $f_1,\dots,f_n \in S$ of nonconstant homogeneous polynomials is a {\em regular sequence} if for all $1 \leq i \leq n$, the map
\[
(-) \times f_i: S/(f_1,\dots,f_{i-1}) \longrightarrow S/(f_1,\dots,f_{i-1})
\]
on the quotient ring $S/(f_1,\dots,f_{i-1})$ is injective. An ideal generated by a regular sequence is called a {\em complete intersection}; quotients by complete intersections have nice Hilbert series. 

\begin{lemma}
    \label{lem:regular-hilbert} (see e.g. \cite[p. 35]{Stanley})
    Suppose $f_1,\dots,f_n \in S$ are homogeneous polynomials which form a regular sequence. Let $I = (f_1,\dots,f_n) \subseteq S$ be the ideal generated by these polynomials. We have
    \begin{equation}
        \Hilb(S/I;t) = \prod_{i=1}^n [\deg(f_i)]_t.
    \end{equation}
\end{lemma}

The following criterion is often useful for proving that a sequence of polynomials is regular. It may be proven using the Nullstellensatz.

\begin{lemma}
    \label{lem:regular-locus}
    Let $f_1,\dots,f_n \in S$ be a sequence of $n$ homogeneous polynomials of positive degree. Let $\overline{\FF}$ be the algebraic closure of $\FF$. Then $f_1,\dots,f_n$ is a regular sequence in $S$ if and only if the polynomial system $f_1= \cdots = f_n = 0$ has only the trivial solution $\zero \in \overline{\FF}^n$.
\end{lemma}

Given an ideal $I \subseteq S$ and an element $f \in S$, the {\em colon ideal} is given by
\begin{equation}
(I:f) := \{ g \in S \,:\,fg \in I \}.
\end{equation}
It is not difficult to show that $(I:f)$ is an ideal of $S$ containing $I$ and that $(I:f)$ is homogeneous whenever both $f$ and $I$ are homogeneous. We need a factorization result on the generating sets of certain colon ideals.

\begin{lemma}
    \label{lem:colon-regular}
    Let $f_1,\dots,f_n \in S$ be a regular sequence of homogeneous polynomials of positive degree. Write $I = (f_1,\dots,f_n) \subseteq S$ for the ideal generated by these polynomials. Suppose we have a factorization $f_1 = f_1' \cdot f_1''$ where $f_1', f_1'' \in S$ are homogeneous. We have the equality
    \begin{equation}
        (I:f_1') = (f_1'', f_2,\dots,f_n)
    \end{equation}
    of ideals in $S$.
\end{lemma}

Lemma~\ref{lem:colon-regular} and its proof should be compared to  \cite[Lem. 2.8]{HHMPT}.

\begin{proof}
    If $\deg(f_1') = 0$ then $f_1' \in \FF^\times$ is a nonzero scalar and $(I:f_1') = (f_1'', f_2,\dots,f_n) = I$. Also, if $\deg(f_1'') = 0$ then $f_1'' \in \FF^\times$ is a nonzero scalar and $(I:f_1') = (f_1'', f_2,\dots,f_n)= S$. We  therefore assume that both $f_1'$ and $f_1''$ have strictly positive degree. 

    Since $f_1 = f_1' \cdot f_1''$, by the definition of $(I:f_1')$ we have the containment of ideals
    \begin{equation}
        \label{eq:prime-ideal-containment}
        (f_1'',f_2,\dots,f_n) \subseteq (I:f_1').
    \end{equation}
    Also be the definition of $(I:f_1')$ we have a short exact sequence
    \begin{equation}
        \label{eq:colon-ses}
        0 \to S/(I:f_1') \xrightarrow{\,\, \times f_1' \, \, } S/I \xrightarrow{\,\, \text{can.} \, \, } S/(I + (f_1')) \to 0
    \end{equation}
    where the first map is induced by multiplication by $f_1'$ and the second map is the canonical projection. Noting that $I + (f_1') = (f_1',f_2,\dots,f_n)$, the short exact sequence \eqref{eq:colon-ses} yields the Hilbert series equality
    \begin{equation}
        \label{eq:hilbert-from-ses}
        \Hilb(S/I;t) = \Hilb(S/(f_1',f_2,\dots,f_n);t) + t^{\deg(f_1')} \cdot \Hilb(S/(I:f_1');t).
    \end{equation}
    Since $f_1'$ and $f_1''$ are nonconstant, Lemma~\ref{lem:regular-locus} implies that both $f_1',f_2,\dots,f_n$ and $f_1'',f_2,\dots,f_n$ are regular sequences. Since $I$ is generated by the regular sequence $f_1,\dots,f_n$, by Lemma~\ref{lem:regular-hilbert} we have the Hilbert series
    \begin{equation}
        \label{eq:colon-hilberts}
        \begin{cases}
            \Hilb(S/I;t) = [\deg(f_1)]_t \cdot [\deg(f_2)]_t \cdots [\deg(f_n)]_t, \\
            \Hilb(S/(f_1',f_2,\dots,f_n);t) = [\deg(f_1')]_t \cdot [\deg(f_2)]_t \cdots [\deg(f_n)]_t, \\
            \Hilb(S/(f_1'',f_2,\dots,f_n);t) = [\deg(f_1'')]_t \cdot [\deg(f_2)]_t \cdots [\deg(f_n)]_t.
        \end{cases}
    \end{equation}
    Since $\deg(f_1) = \deg(f_1') + \deg(f_1'')$ we have $[\deg(f_1)]_t = [\deg(f_1')]_t + t^{\deg(f_1')}[\deg(f_1'')]_t$ so that \eqref{eq:hilbert-from-ses} and \eqref{eq:colon-hilberts} force
    \begin{equation}
        \Hilb(S/(I:f_1');t) =  \Hilb(S/(f_1'',f_2,\dots,f_n);t).
    \end{equation}
    In turn, the containment of ideals \eqref{eq:prime-ideal-containment} forces the desired equality of ideals $(I:f_1') = (f_1'',f_2,\dots,f_n)$ and the proof is complete.
\end{proof}

\subsection{Divided power rings and inverse systems}
Let $I \subseteq S$ be a homogeneous ideal. The {\em inverse system} $I^\perp$ associated to $I$ is a useful tool for understanding the quotient ring $S/I$ without the direct use of cosets. In order to define inverse systems over fields $\FF$ of potentially-nonzero characteristic, we need the terminology of {\em divided power rings}. We refer the reader to e.g. \cite{RR} for a more leisurely overview.

The rank $n$ {\em divided power ring} $\DD$ over $\FF$ is the free $\FF$-module generated by the symbols $y_1^{(a_1)} \cdots y_n^{(a_n)}$ for $a_1,\dots,a_n \geq 0$. We have a vector space direct sum decomposition $\DD = \bigoplus_{i \geq 0} \DD_i$ where 
\[
\DD_i := \FF \cdot \left\{ y_1^{(a_1)} \cdots y_n^{(a_n)} \,:\, a_1 + \cdots + a_n  = i \right\}.
\]
We define a multiplication on $\DD$ by the $\FF$-bilinear extension of
\begin{equation}
    (y_1^{(a_1)} \cdots y_n^{(a_n)}) 
    \cdot (y_1^{(b_1)} \cdots y_n^{(b_n)}) 
    := {a_1 + b_1 \choose a_1, b_1} \cdots {a_n + b_n \choose a_n, b_n} 
    \cdot y_1^{(a_1 + b_1)} \cdots y_n^{(a_n + b_n)}
\end{equation}
where the binomial coefficients ${a_i + b_i \choose a_i, b_i}$ are regarded as elements of $\FF$ via
\[
{a_i + b_i \choose a_i, b_i} := \overbrace{1 + \cdots + 1}^{\frac{(a_i+b_i)!}{a_i! \cdot b_i!}} \in \FF.
\]
This endows $\DD = \bigoplus_{i \geq 0} \DD_i$ with the structure of a graded $\FF$-algebra. Roughly speaking, the element $y_1^{(a_1)} \cdots y_n^{(a_n)} \in \DD$ should be thought of as ``$\frac{y_1^{a_1} \cdots y_n^{a_n}}{a_1! \cdots a_n!}$", even in the positive characteristic setting where the product $a_1! \cdots a_n!$ may vanish in $\FF$.

There is an action $\odot: S \otimes \DD \to \DD$ of $S$ on $\DD$ given by 
\begin{equation}
    (x_1^{b_1} \cdots x_n^{b_n}) \odot (y_1^{(a_1)} \cdots y_n^{(a_n)}) := y_1^{(a_1 - b_1)} \cdots y_n^{(a_n - b_n)}
\end{equation}
and $\FF$-bilinear extension. Here we interpret $y_1^{(a_1 - b_1)} \cdots y_n^{(a_n - b_n)} := 0$ if $a_i - b_i < 0$ for some $i$. In particular, for any fixed degree $i$ we have a bilinear map
\begin{equation}
    \langle-,-\rangle: S_i \times \DD_i \longrightarrow \FF \quad \quad \langle f,g\rangle := f \odot g.
\end{equation}
It is not difficult to check that the map $\langle -, - \rangle$ is a perfect pairing between the finite-dimensional $\FF$-vector spaces $S_i$ and $\DD_i$. That is, for all nonzero $f \in S_i$ there exists $g \in \DD_i$ so that $\langle f,g\rangle \neq 0$.

Let $I \subseteq S$ be a homogeneous ideal. The {\em inverse system} (or {\em harmonic space}) associated to $I$ is the vector subspace $I^\perp \subseteq \DD$ given by
\begin{equation}
    I^\perp := \{ g \in \DD \,:\, f \odot g = 0 \text{ for all } f \in I \}.
\end{equation}
It can be shown that $I^\perp$ is a graded subspace of $\DD$ which is closed under the $\odot$-action of $S$. Furthermore, since $\langle-,-\rangle$ is a perfect pairing one has
\begin{equation}
    \Hilb(S/I;t) = \Hilb(I^\perp;t).
\end{equation}

\subsection{Superspace conventions} As in the introduction, we write $\Omega$ for the rank $n$ {\em superspace ring} of regular differential forms on $V := \FF^n$. We have 
\begin{equation}
    \Omega = \Sym(V^*) \otimes \wedge(V^*) \cong \FF[x_1,\dots,x_n] \otimes \wedge \{ \theta_1,\dots,\theta_n\}.
\end{equation}
We introduce the notation
\begin{equation}
    E := \wedge(V^*) \cong \wedge \{ \theta_1,\dots,\theta_n\}
\end{equation}
for the rank $n$ free exterior algebra so that
\begin{equation}
    \Omega = S \otimes E.
\end{equation}
The $\FF$-algebra $\Omega$ is bigraded via $\Omega_{i,j} := S_i \otimes E_j$.

For a subset $J = \{ j_1 < \cdots < j_k \} \subseteq [n]$, the {\em exterior monomial} is the increasing product
\begin{equation}
    \theta_J := \theta_{j_1} \wedge \cdots \wedge \theta_{j_k}
\end{equation}
of the corresponding exterior variables.  We have a direct sum decomposition
\begin{equation}
    \Omega = \bigoplus_{J \subseteq [n]} S \cdot \theta_J.
\end{equation}
The Gale order $\leq_\Gale$ induces a partial order on the summands of this decomposition with the same exterior degree.

Let $\partial_i: S \to S$ be the operator of partial differentiation with respect to $x_i$. We extend $\partial_i$ to a superspace operator $\partial_i: \Omega \to \Omega$ by identifying $\Omega = S \otimes E$ and acting on $S \otimes E$ by $\partial_i \otimes \mathrm{id}_E$. The {\em Euler derivation} is the map
\begin{equation}
    d: \Omega \longrightarrow \Omega
\end{equation}
given by the formula
\begin{equation}
    d(f) := \sum_{i=1}^n \partial_i f \wedge \theta_i.
\end{equation}

\subsection{Modular (super)invariant theory} In this subsection we specialize to the finite field $\FF = \FF_q$ where $q$ is a prime power. Let $\alpha = (\alpha_1,\dots,\alpha_n)$ be a list of $n$ nonnegative integers. The {\em Moore matrix} associated to $\alpha$ is the $S$-matrix with entries
\begin{equation}
    M_\alpha := \begin{pmatrix}
        x_1^{q^{\alpha_1}} & \cdots  & x_1^{q^{\alpha_n}} \\
        \vdots & \ddots & \vdots  \\
        x_n^{q^{\alpha_1}} & \cdots & x_n^{q^{\alpha_n}}
    \end{pmatrix}.
\end{equation}
The determinant $\det M_\alpha$ is a polynomial in $S$. 

L. E. Dickson introduced \cite{Dickson} the elements $D_{n,i}$ for $i = 0,1,\dots,n-1$ given by
\begin{equation}
    D_{n,i} := \frac{\det M_{(n,n-1,\dots,i+1,i-1,\dots,1,0)}}{\det M_{\rho}}
\end{equation}
where 
\begin{equation}
    \rho := (n-1,\dots,1,0)
\end{equation}
is the `staircase' sequence.
Dickson proved that $D_{n,i}$ is a polynomial of degree $q^n - q^i$ in $S$ (rather than merely a rational expression) which is invariant under the action of $GL_n(\FF_q)$. He proved further that the $GL_n(\FF_q)$-invariant subring of $S$ is freely generated by $D_{n,0}, D_{n,1},\dots,D_{n,n-1}$. Introducing a new variable $t$, the Dickson invariants are characterized by the equation
\begin{equation}
    \label{eq:dickson-defining}
    \sum_{a_1,\dots,a_n \in \FF_q} (t + a_1 x_1 + \cdots + a_n x_n) = t^{q^n} + \sum_{i=0}^{n-1} D_{n,i} \cdot t^{q^i}.
\end{equation}

Suppose $G$ is a subgroup of $GL_n(\FF_q)$ which contains the special linear group $SL_n(\FF_q)$. The group $G$ is determined by the order
\[
e = | \{ \det(g) \,:\, g \in G \} |
\]
of the image of the determinant homomorphism $\det: G \to \FF^\times$. The following generalization of Dickson's Theorem is well-known.

\begin{theorem}
    \label{thm:between-polynomial-generators}
    Let $G$ be a subgroup of $GL_n(\FF_q)$ which contains $SL_n(\FF_q)$. The $\FF_q$-algebra $S^G$ of polynomial $G$-invariants is freely generated by 
    \[
    (\det M_\rho)^e, \quad D_{n,1}, \quad D_{n,2}, \quad \dots \quad ,\quad D_{n,n-1}
    \]
    where $\rho  = (n-1,\dots,1,0)$ and $e = |\{\det(g)\,:\, g \in G\}|$.
\end{theorem}

Using $\doteq$ for equality up to a nonzero scalar, it can be shown that
\begin{equation}
    D_{n,0} \doteq (\det M_\rho)^{q-1}
\end{equation}
so that Theorem~\ref{thm:between-polynomial-generators} generalizes Dickson's result.

The superspace invariant theory for matrix groups $G$ containing $SL_n(\FF_q)$ (among other subgroups of $GL_n(\FF_q)$) was worked out by Hartmann and Shepler \cite{HS}.  In particular, suppose $SL_n(\FF_q) \subseteq G$ with $e = |\{\det(g)\,:\,g \in G\}|$. The {\em twisted wedge product} $\curlywedge$ associated to $G$ is given by
\begin{equation}
    f \curlywedge g := \frac{f \wedge g}{(\det M_\rho)^e}
\end{equation}
for $G$-invariant superspace elements $f, g\in \Omega^G$. Hartmann and Shepler proved \cite{HS} that $\Omega^G$ is closed under the bilinear multiplication $\curlywedge$. 

Twisted wedge products were introduced in the characteristic zero setting by Shepler \cite{Shepler} in her study of differential semi-invariants for complex reflection groups. In the positive characteristic setting, Hartmann and Shepler \cite{HS} proved the following result. Recall that $d: \Omega \to \Omega$ is the Euler derivation.

\begin{theorem}
    \label{thm:HS}
    {\em (Hartmann--Shepler \cite{HS})} Let $G$ be a subgroup of $GL_n(\FF_q)$ which contains $SL_n(\FF_q)$ and let $e = |\{\det(g) \,:\, g \in G \}|$. The $G$-invariant subring $\Omega^G \subseteq \Omega$ is freely generated as a $S^G$-module by the $n$ elements
    \[
    \frac{d(D_{n,i})}{(\det M_\rho)^{q-1-e}} \quad \quad 0 \leq i \leq n-1
    \]
    with respect to the twisted wedge product $\curlywedge$ associated to $G$.
\end{theorem}

Implicit in Theorem~\ref{thm:HS} is the statement that $(\det M_\rho)^{q-1-e}$ divides $d(D_{n,i})$ in the superspace ring $\Omega$ and that the quotient lies in $\Omega^G$. This divisibility and membership is established on \cite[p. 11]{HS}.

\section{Hilbert Series Upper Bound}
\label{sec:Upper}

Throughout this section, we work over the finite field $\FF_q$ where $q$ is a prime power and write $G = GL_n(\FF_q)$ for the full finite general linear group. Our goal is to bound the Hilbert series $\Hilb(SR_G;t,z)$ of the $G$-superspace coinvariant ring $SR_G$ from above.  In order to do this, we need to recall a $GL_n(\FF_q)$-analog of the Schur polynomial due to Macdonald \cite{MacdonaldVariations}.

\subsection{Macdonald's 7th variation of the Schur polynomial} Let $\lambda = (\lambda_1  \geq  \cdots \geq \lambda_n \geq 0)$ be a partition with at most $n$ parts. Macdonald introduced \cite{MacdonaldVariations} the quotient of Moore matrix determinants
\begin{equation}
\label{eq:7-definition}
    s^{(q)}_\lambda = s^{(q)}_\lambda(x_1,x_2,\dots,x_n) := \frac{\det M_{\lambda + \rho}}{\det M_\rho}
\end{equation}
where $\rho = (n-1,\dots,1,0)$ and $\lambda  + \rho$ is componentwise addition. Equation~\eqref{eq:7-definition} is a $GL_n(\FF_q)$-analog of Cauchy's original bialternant definition of the usual Schur polynomial. Macdonald proved that $s^{(q)}_\lambda$ is a polynomial in $S$ (rather than merely a rational expression) and that $s^{(q)}_\lambda \in S^G$ is $G$-invariant.\footnote{In contrast to the characteristic 0 case, 
the $s^{(q)}_\lambda$ do {\bf not} span $S^G$ over $\FF_q$. This is even true when $n=1$.}

When $\lambda = (1^{n-i},0^i)$, the polynomial $s^{(q)}_\lambda$ reduces to the Dickson invariant $D_{n,i}$, so Dickson invariants can be thought of as `$GL_n(\FF_q)$-elementary symmetric polynomials'. The `complete homogeneous case' where $\lambda$ has a single part plays a critical role in the next subsection.

\subsection{A superspace identity} 
This subsection proves a superspace identity (Lemma~\ref{lem:syzygy}) which will prove useful in our study of $SR_G$. Before stating this identity, we introduce some terminology.

\begin{definition}
    \label{def:P-Q}
    For $1 \leq i \leq n$, define a polynomial $P_i(t) \in S[t]$ by
    \begin{equation}
        P_i(t) := \prod_{a_i, a_{i+1},\dots,a_n \in \FF_q} (t + a_i x_i + a_{i+1} x_{i+1} +  \cdots + a_n x_n).
    \end{equation}
    Also define $P_{n+1}(t) := t$. For $1 \leq i \leq n$, define $Q_i \in S$ by
    \begin{equation}
        Q_i := P_{i+1}(x_i) = 
        \prod_{a_{i+1},\dots,a_n \in \FF_q} (x_i + a_{i+1} x_{i+1} +  \cdots + a_n x_n).
    \end{equation}
\end{definition}

The polynomials $Q_i \in S$ were proven by M\`ui \cite{Mui} to freely generate the invariant ring $S^U$ where $U \subseteq GL_n(\FF_q)$ is the  subgroup of lower triangular matrices with $1$'s on the diagonal. We record some simple facts about the $P_i(t)$ and $Q_i$.

\begin{lemma}
    \label{lem:simple-PQ}
    Let $D_{n,i} \in S$ be the Dickson invariant.
    \begin{enumerate}
        \item We have $D_{n,0} = \prod_{i=1}^n Q_i^{q-1}.$
        \item The polynomial $P_i(t) \in S[t]$ expands as
        \[
        P_i(t) = t^{q^{n-i+1}} + \sum_{j=0}^{n-i} D_{n-i+1,j}(x_i,x_{i+1},\dots,x_n) \cdot t^{q^j}
        \]
        where $D_{n-i+1,j}(x_i,x_{i+1},\dots,x_n)$ is the (symmetric) polynomial $D_{n-i+1,j}$ in the variables $x_i, x_{i+1},\dots,x_n$.
        \item We have $P_i(x_j) = 0$ for all $j \geq i$.
        \item In the ring $S[u,v]$ there hold the identities
        \[
        P_i(u+v) = P_i(u) + P_i(v) \quad \text{and} \quad P_i(a \cdot u) = a \cdot P_i(u)
        \]
        for any scalar $a \in \FF_q$.
        \item We have the recursion
        \[
        P_i(t) = P_{i+1}(t)^q - Q_i^{q-1} \cdot P_{i+1}(t).
        \]
    \end{enumerate}
\end{lemma}

\begin{proof}
    (1) and (2) both follow from Equation~\eqref{eq:dickson-defining} and Definition~\ref{def:P-Q}.  If $j \geq i$ then $t - x_j$ is among the factors of $P_i(t)$ and (3) follows.
    (4) follows from the explicit formula in (2) and the fact that the Frobenius map $x \mapsto x^q$ is an endomorphism of the ring $S[u,v]$ which fixes $\FF_q$.

    For (5), we start with the identity in $\FF_q[x,y]$ given by 
    \begin{equation}
    \label{eq:two-variable-identity}
        \prod_{a \in \FF_q} (x + a \cdot  y) = x^q - x \cdot y^{q-1}.
    \end{equation}
    Equation~\eqref{eq:two-variable-identity} follows from the standard identity $\prod_{a \in \FF_q} (t-a) = t^q - t$ by substituting $t \to \frac{x}{y}$ and multiplying by $y^q$. Definition~\ref{def:P-Q} gives
\begin{align}
    P_i(t) &= \prod_{a \in \FF_q} \prod_{a_{i+1},\dots,a_n \in \FF_q} (t + a x_i + a_{i+1} x_{i+1} + \cdots + a_n x_n) \\ 
    &= \prod_{a \in \FF_q} P_{i+1}(t + a \cdot x_i)
    \\&= \prod_{a \in \FF_q} (P_{i+1}(t) + a \cdot P_{i+1}(x_i))
     \\&= \prod_{a \in \FF_q} (P_{i+1}(t) + a \cdot Q_i)
\end{align}
where the penultimate equality uses (4). Applying \eqref{eq:two-variable-identity} proves (5).
\end{proof}

We are ready to state the fundamental superspace syzygy which will bound $SR_G$ from above. The relevant identity is homogeneous of exterior degree 1 and involves Macdonald's 7th Schur function variation $s^{(q)}_\lambda$ evaluated on partial variable sets. For $m \geq 0$ we abbreviate $s_m^{(q)} := s_{(m)}^{(q)}$.

\begin{lemma}
    \label{lem:syzygy}
    Let $1 \leq i \leq n$ and consider the element $f_i \in \Omega$ given by 
    \begin{equation}
        f_i := \sum_{j=1}^i s_{i-j}^{(q)}(x_i, x_{i+1}, \dots, x_{n-1}, x_n) \cdot  d (D_{n,n-j}).
    \end{equation}
    For $i' < i$, the coefficient of $\theta_{i'}$ in $f_i$ vanishes. On the other hand, we have
    \begin{equation}
    \text{coefficient of $\theta_i$ in $f_i$} \doteq  
    \frac{D_{n,0}}{Q_i} = \frac{Q_1^{q-1} \cdots Q_n^{q-1}}{Q_i}.
    \end{equation}
\end{lemma}

The proof of Lemma~\ref{lem:syzygy} is technical and may be skipped on a first reading.

\begin{proof}
    Lemma~\ref{lem:simple-PQ} (5) gives the equation in $S[t]$:
    \begin{equation}
        \label{eq:syzygy-one}
        P_i(t) = P_{i+1}(t)^q - Q_i^{q-1} \cdot P_{i+1}(t).
    \end{equation}
    We apply the Euler operator $d$ to both sides of Equation~\eqref{eq:syzygy-one}. Since $q= 0$ in $\FF_q$, the product and chain rules yield the following equation in $\Omega[t]$:
    \begin{equation}
        \label{eq:syzygy-two}
        d P_i(t) =  Q_i^{q-2} \cdot dQ_i \cdot P_{i+1}(t) - Q_i^{q-1} \cdot dP_{i+1}(t).
    \end{equation}
    
    Iterating Equation~\eqref{eq:syzygy-two}, we obtain the following expansion of $d P_1(t)$ in terms of the $P_i(t)$'s:
    \begin{equation}
    \label{eq:syzygy-three}
        d P_1(t) = \sum_{j=0}^{n-1} d D_{n,j} \cdot t^{q^j} 
        = \sum_{r=1}^n (-1)^{r-1} \left( \prod_{k < r} Q_k^{q-1} \right) \cdot Q_{r}^{q-2} \cdot d Q_r \cdot P_{r+1}(t)
    \end{equation}
    where the first equality is justified by Lemma~\ref{lem:simple-PQ} (2).
    Let $1 \leq j \leq n$. Extracting the coefficient of $t^{q^{n-j}}$ in Equation~\eqref{eq:syzygy-three} gives
    \begin{equation}
    \label{eq:syzygy-four}
        d D_{n,n-j} = \sum_{r \leq j} (-1)^{r-1}  
         \cdot \{t^{q^{n-j}}\} P_{r+1}(t)  \times \left( \prod_{k < r} Q_k^{q-1} \right) \cdot  (Q_{r}^{q-2} \cdot d Q_r)
    \end{equation}
    where 
    $\{t^{q^{n-j}}\} P_{r+1}(t)$ is the coefficient of $t^{q^{n-j}}$ in $P_{r+1}(t)$. By Lemma~\ref{lem:simple-PQ} (2), the coefficient of $t^{q^{n-j}}$ in $P_{r+1}(t)$ vanishes for $r > j$, hence the sum in \eqref{eq:syzygy-four} ranging over $r \leq j$.

    The superspace element $f_i \in \Omega$ is defined as 
    \begin{equation}
      \label{eq:syzygy-five}
      f_i = \sum_{j \leq i} s_{i-j}^{(q)}(x_i, x_{i+1}, \dots, x_{n-1}, x_n) \cdot  d D_{n,n-j}
    \end{equation}
    Plugging \eqref{eq:syzygy-four} into \eqref{eq:syzygy-five} yields
    \begin{equation}
        \label{eq:syzygy-six}
        f_i = \sum_{r \leq i} C_{i,r} \times \left( \prod_{k < r} Q_k^{q-1} \right) \cdot  (Q_{r}^{q-2} \cdot d Q_r)
    \end{equation}
    where the coefficient $C_{i,r} \in S$ is given by
    \begin{equation}
        \label{eq:syzygy-seven}
        C_{i,r} := (-1)^{r-1} \cdot \sum_{j=r}^i s^{(q)}_{i-j}(x_i, x_{i+1},\dots,x_n) \cdot \{t^{q^{n-j}}\} P_{r+1}(t). 
    \end{equation}
    The following lemma gives the value of $C_{i,r}$ for all $r \leq i$. 

    \begin{lemma}
        \label{lem:value-of-C}
        For $r \leq i$, the value of $C_{i,r}$ is given by
        \begin{equation}
            C_{i,r} = (-1)^{r-1} \cdot \sum_{j=r}^i s^{(q)}_{i-j}(x_i, x_{i+1},\dots,x_n) \cdot 
           s_{(1^{j-r})}^{(q)}(x_{r+1}, x_{r+2}, \dots, x_n)
            = \begin{cases} (-1)^{r-1} & r = i, \\
        0 & r < i. \end{cases}
        \end{equation}
    \end{lemma}

    \begin{proof} (of Lemma~\ref{lem:value-of-C})
    Let $K := \FF_q(x_1,\dots,x_n)$ be the fraction field of $S$.  
    We regard $P_i(t)$ as a polynomial in $K[t]$. By Lemma~\ref{lem:simple-PQ} (2), $P_i(t) \in K[t]$ is monic  of degree $q^{n-i+1}$ and only involves the $t$-powers $t^{q^\ell}$ for $0 \leq \ell \leq n-i+1$.
    For any $m \geq 0$, we claim that there are unique coefficients $c_{m,\ell} \in K$ so that 
    \begin{equation}
    \label{eq:claimed-reduction}
        t^{q^{n-i+m}} \equiv \sum_{\ell = 0}^{n-i} c_{m,\ell} \cdot t^{q^\ell} \mod P_i(t).
    \end{equation}
    Indeed, since $\deg P_i(t) = q^{n-i+1}$, the uniqueness of the $c_{m,\ell}$ will follow from their existence. 
    When $m = 0$ we have 
    \[
    c_{0,\ell} = \begin{cases}
        1 & \ell = n-i, \\
        0 & \text{otherwise}.
    \end{cases}
    \]
    When $m = 1$ we have $c_{1,\ell} = -D_{n-i+1,\ell}(x_i,\dots,x_n)$ by Lemma~\ref{lem:simple-PQ} (2).
    Observe that the Frobenius ring map $x \mapsto x^q$ preserves the ideal $(P_i(t)) \subseteq K[t]$.
    Consequently, we have
    \begin{equation}
        \label{eq:frobenius-implication}
        f(t) \equiv g(t) \mod P_i(t) \quad \Rightarrow \quad f(t)^q \equiv g(t)^q \mod P_i(t) \quad \text{for all $f(t),g(t) \in K[t].$}
    \end{equation}
    For $m > 1$, iterating \eqref{eq:frobenius-implication} $m-1$ times to the modular equation
    \begin{equation}
        t^{q^{n-i+1}} \equiv \sum_{\ell = 0}^{n-i} c_{1,\ell} \cdot t^{q^\ell} \mod P_i(t)
    \end{equation}
    yields
    \begin{equation}
    \label{eq:modular-reduction}
        t^{q^{n-i+m}} = (t^{q^{n-i+1}})^{q^{m-1}} \equiv \sum_{\ell = 0}^{n-i} c_{1,\ell}^{q^{m-1}} \cdot t^{q^{\ell + m-1}} \mod P_i(t).
    \end{equation}
    Every power $t^{q^{\ell+m-1}}$ appearing on the right-hand side of the congruence \eqref{eq:modular-reduction} satisfies \[\ell + m - 1 < n-i+m\]
    and by induction on $m$ each term $c_{1,\ell}^{q^{m-1}} \cdot t^{q^{\ell + m-1}}$ in this sum may be reduced as in \eqref{eq:claimed-reduction}.

    We use Cramer's Rule to study the coefficient $c_{m,n-i} \in K$ appearing in the modular equation \eqref{eq:claimed-reduction}.  By Lemma~\ref{lem:simple-PQ} (3), we have $P_i(x_j) = 0$ whenever $j \geq i$. Equation~\eqref{eq:claimed-reduction} therefore gives rise to the linear system over $K$ given by
    \begin{equation}
    \begin{pmatrix}
        x_i^{q^{n-i}} & x_i^{q^{n-i-1}} & \cdots & x_i \\
        x_{i+1}^{q^{n-i}} & x_{i+1}^{q^{n-i-1}} & \cdots & x_{i+1} \\
         & & \ddots & \\
        x_n^{q^{n-i}} & x_n^{q^{n-i-1}} & \cdots & x_n
    \end{pmatrix} 
    \begin{pmatrix}
        c_{m,n-i} \\ c_{m,n-i-1} \\ \vdots \\ c_{m,0}
    \end{pmatrix} =
    \begin{pmatrix}
        x_i^{q^{n-i+m}} \\ x_{i+1}^{q^{n-i+m}} \\ \vdots \\ x_n^{q^{n-i+m}}
    \end{pmatrix}.
    \end{equation}
    In particular, when $\ell = n-i$ Cramer's Rule and the definition of $s^{(q)}_m$ give
    \begin{equation}
        c_{m,n-i} = s_m^{(q)}(x_i, x_{i+1}, \dots, x_n) \in S.
    \end{equation}
    For any polynomial $F(t) = \sum_{\ell \geq 0} b_\ell \cdot t^{q^\ell} \in K[t]$ only involving the powers $t^{q^\ell}$ for $\ell \geq 0$, we therefore have a reduction
    \begin{equation}
        F(t) \equiv \sum_{\ell = 0}^{n-i} c_{F(t),\ell} \cdot t^{q^\ell} \mod P_i(t)
    \end{equation}
    for some $c_{F(t),\ell} \in K$ so that 
    \begin{equation}
    \label{eq:value-of-c}
        c_{F(t),n-i} = \sum_{\ell \geq n-i} b_\ell \cdot s_{\ell - n + i}^{(q)}(x_i, x_{i+1},\dots,x_n).
    \end{equation}

     Given $r \leq i$, we apply Equation~\eqref{eq:value-of-c} to $F(t) = P_{r+1}(t)$. For $r < i$, we have $P_{i}(t) \mid P_{r+1}(t)$ in $S[t]$ so that $P_{r+1}(t) \equiv 0 \mod P_{i}(t)$. Therefore
    \begin{align}
        0 &= c_{P_{r+1}(t),n-i} \\
        &= \sum_{\ell \geq n-i} s_{\ell-n+i}^{(q)}(x_i, \dots, x_n) \cdot \{t^{q^\ell}\} P_{r+1}(t)
    \end{align}
    where the second equality used Equation~\eqref{eq:value-of-c}. Since we have
    \begin{equation}
        \{ t^{q^{n-j}} \} P_{r+1}(t) = D_{n-r,n-j}(x_{r+1}, \dots, x_n) = s^{(q)}_{(1^{j-r})}(x_{r+1},\dots,x_n),
    \end{equation}
    making the change of variables $\ell \mapsto n - j$ yields Lemma~\ref{lem:value-of-C} in the case $r < i$. On the other hand, if $r = i$ we have
    \begin{equation}
        C_{i,i} = (-1)^{i-1} \cdot s_0^{(q)}(x_i, x_{i+1},\dots,x_n) \cdot \{ t^{q^{n-i}} \} P_{i+1}(t) = (-1)^{i-1} \cdot 1 \cdot 1.
    \end{equation}
    This completes the proof of Lemma~\ref{lem:value-of-C}.
    \end{proof}
    


     We use Lemma~\ref{lem:value-of-C} to complete the proof of Lemma~\ref{lem:syzygy} as follows.
     Since $d Q_i$ only involves the exterior variables $\theta_i, \theta_{i+1}, \dots, \theta_n$, by \eqref{eq:syzygy-six} and Lemma~\ref{lem:value-of-C}  the coefficient of $\theta_{i'}$ in $f_i$ vanishes whenever $i' < i$. Furthermore, again by \eqref{eq:syzygy-six}, the coefficient of $\theta_i$ in $f_i$ is
    \begin{equation}
    \label{eq:desired-coefficient}
       (-1)^{i-1} \cdot  \prod_{k < i} Q_k^{q-1} \cdot Q_i^{q-2} \cdot \partial_i Q_i \doteq 
        \prod_{k < i} Q_k^{q-1} \cdot Q_i^{q-2} \cdot \prod_{k > i} Q_k^{q-1} = \frac{D_{n,0}}{Q_i}
    \end{equation}
    where $\partial_i$ is partial differentiation with respect to $x_i$.
    The $\doteq$ in \eqref{eq:desired-coefficient} uses
    \begin{align}
        \partial_i Q_i &=  \left[ \frac{\partial}{\partial t} P_{i+1}(t) \right]_{t \to x_i} \\ &= \left[ \frac{\partial}{\partial t} \left( t^{q^{n-i}} +  \sum_{j=0}^{n-i-1} D_{n-i,j}(x_{i+1},\dots,x_n) \cdot t^{q^j} \right) \right]_{t \to x_i} \\
        &= D_{n-i,0}(x_{i+1},\dots,x_n) \\
        &\doteq \prod_{k > i} Q_k^{q-1}
    \end{align}
    where the first equality is Definition~\ref{def:P-Q}, the second is Lemma~\ref{lem:simple-PQ} (2), the third follows because $q^j = 0$ in $\FF_q$ for all $j > 0$,  the fourth is Lemma~\ref{lem:simple-PQ} (1). (When $i=n$ we interpret $D_{0,0}$ in the empty variable set to be 1.) This completes the proof of Lemma~\ref{lem:syzygy}. 
    \end{proof}

\begin{remark}
    Lemma~\ref{lem:value-of-C} is reminiscent of the classical symmetric function theory involution
    \[
    \sum_{i=0}^n (-1)^i e_i h_{n-i} = \begin{cases}
        1 & n = 0,\\
        0 & n > 0.
    \end{cases}
    \]
    Here $e_i$ is the elementary symmetric function and $h_{n-i}$ is the complete homogeneous symmetric function. Macdonald proved \cite[(7.9)]{MacdonaldVariations} a different $GL_n(\FF_q)$-convolution identity for the $s^{(q)}_m$ and $s^{(q)}_{(1^m)}$ involving the fixed variable set $\{x_1,\dots,x_n\}$.
\end{remark}

\subsection{Hilbert series upper bound}
In this subsection we apply the technical Lemma~\ref{lem:syzygy} to obtain an upper bound on the bigraded Hilbert series of $SR_G$. The twisted wedge product appearing in the Hartmann--Shepler Theorem~\ref{thm:HS} plays a key role.

For any $J \subseteq [n]$, define $\widetilde{D}_{n,J} \in S$ to be the polynomial
\begin{equation}
    \widetilde{D}_{n,J} := \prod_{j \in J} Q_j^{q-2} \times \prod_{\substack{1 \leq i \leq n \\ i \notin J}} Q_i^{q-1}.
\end{equation}
Observe that when $J = \{i\}$ is a singleton, the polynomial $\widetilde{D}_{n,\{i\}}$ is (up to a nonzero scalar) the coefficient of $\theta_i$ appearing in the polynomial $f_i$ of Lemma~\ref{lem:syzygy}. For any subset $J \subseteq [n]$, define
\begin{equation}
    Q_J := \prod_{j \in J} Q_j \in S.
\end{equation}
Lemma~\ref{lem:simple-PQ} (1) implies
\begin{equation}
\label{eq:tilde-factor}
    \widetilde{D}_{n,J} \times Q_J = D_{n,0}.
\end{equation}
Taking degrees the factorization \eqref{eq:tilde-factor} gives
\begin{equation}
    \label{eq:degree-of-tildeD}
    \deg \widetilde{D}_{n,J} = q^n - 1 - \sum_{j \in J} q^{n-j}.
\end{equation}
The following purely commutative quotient ring will be relevant in our study of $SR_G$.

\begin{lemma}
    \label{lem:tilde-quotient}
    Let $J \subseteq [n]$ and let $I_G = (D_{n,0}, D_{n,1},\dots,D_{n,n-1}) \subseteq S$ be the ideal generated by $G$-invariants with vanishing constant term. We have the equality of ideals  in $S$:
    \begin{equation}
    \label{eq:colon-ideal-identification}
        \left(
                I_G:Q_J
        \right) = (\widetilde{D}_{n,J}, D_{n,1}, D_{n,2}, \dots, D_{n,n-1}).
    \end{equation} 
    Furthermore, the quotient of $S$ by either of these ideals has Hilbert series
    \begin{align}
        \Hilb\left( \frac{S}{\left(
                I_G:Q_J
        \right)}; t  \right) &= 
        \Hilb( S/(\widetilde{D}_{n,J}, D_{n,1}, D_{n,2}, \dots, D_{n,n-1});t) \\
        &= \prod_{i=1}^{n-1} [q^n - q^i] \times \left[ q^n - 1 - \sum_{j \in J} q^{n-j} \right]_t.
    \end{align}
\end{lemma}

The Hilbert series in Lemma~\ref{lem:tilde-quotient} is exactly the contribution of $J \subseteq [n]$ to the claimed formula for $\Hilb(SR_G;t,z)$ in Equation~\eqref{eq:intro-hilbert}, without the factor of $z^{|J|}$. We will see that this is no accident. 

\begin{proof}
    The identification \eqref{eq:colon-ideal-identification} is a consequence of Lemma~\ref{lem:colon-regular}. The Hilbert series formula follows from Equation~\eqref{eq:degree-of-tildeD} and Lemma~\ref{lem:regular-hilbert}. 
\end{proof}

We are ready to bound the Hilbert series of $SR_G$ from above. Given polynommials $f, g \in \ZZ[t,z]$, we write $f \leq g$ to mean that $g-f \in \ZZ[t,z]$ has nonnegative coefficients. Observe that the following bound matches Equation~\eqref{eq:intro-hilbert}.

\begin{lemma}
    \label{lem:hilbert-upper-bound}
    Let $G = GL_n(\FF_q)$. The Hilbert series $\Hilb(SR_G;t,z)$ of the superspace coinvariant ring $SR_G$ satisfies
    \begin{equation}
        \Hilb(SR_G;t,z) \leq \prod_{i=1}^{n-1} [q^n - q ^i]_t \times \sum_{J \subseteq [n]} z^{|J|} \cdot \left[ q^n - 1 - \sum_{j \in J} q^{n-j} \right]_t.
    \end{equation}
\end{lemma}

\begin{proof}
    Let $\Omega^G_{*,1} \subseteq \Omega^G$ be the subspace of $\Omega^G$ consisting of $G$-invariants of homogeneous exterior degree 1.
    Recall the Hartmann--Shepler twisted wedge product $\curlywedge$ associated to $G = GL_n(\FF_q)$ appearing in Theorem~\ref{thm:HS}. For all $1 \leq i \leq n$, Lemma~\ref{lem:syzygy} furnishes an element $f_i \in \Omega$. By Theorem~\ref{thm:HS} and the explicit formula in Lemma~\ref{lem:syzygy} we have
    \begin{equation}
    \label{eq:fi-membership}
        f_i \in S \cdot \Omega^G_{*,1} \quad \text{for all } 1 \leq i \leq n.
    \end{equation}
    For any $J = \{j_1 < \cdots < j_k \} \subseteq [n]$, we define
    \begin{equation}
        f_J := f_{j_1} \curlywedge \cdots \curlywedge f_{j_k}.
    \end{equation}
    Thanks to the membership \eqref{eq:fi-membership}, Theorem~\ref{thm:HS} and Lemma~\ref{lem:syzygy} imply that 
    \begin{equation}
        \label{eq:fJ-membership}
        f_J \in SI_G = (\Omega^G_+).
    \end{equation}
    Lemma~\ref{lem:syzygy} and the definition of $\curlywedge$ give the following two facts about $f_J$:
    \begin{itemize}
        \item If $K \subseteq [n]$ satisfies $|K| = |J|$, the coefficient of $\theta_K$ in $f_J$ vanishes unless $J \leq_\Gale K$.
        \item Up to a nonzero scalar, the coefficient of $\theta_J$ in $f_J$ is $\widetilde{D}_{n,J}$.
    \end{itemize}

    Fix $J \subseteq [n]$. The superspace ideal $SI_G$ contains the elements
    \[
    f_J, \quad D_{n,1} \cdot \theta_J, \quad D_{n,2} \cdot \theta_J, \quad \dots \quad , \quad D_{n,n-1} \cdot \theta_J.
    \]
    The unique Gale-minimal exterior monomial appearing in each of these elements is $\theta_J$. The coefficients of $\theta_J$ in these superspace elements are (up to nonzero scalar)
    \[
    \widetilde{D}_{n,J}, \quad D_{n,1}, \quad D_{n,2}, \quad \dots \quad , \quad D_{n,n-1}.
    \]
    Let $\BBB_J$ be any homogeneous $\FF_q$-vector space basis of the commutative quotient ring \[S/(\widetilde{D}_{n,J}, D_{n,1}, D_{n,2}, \dots, D_{n,n-1}).\]
    Any polynomial $g \in S$ has a unique expression of the form
    \begin{equation}
        g \equiv \sum_{b \in \BBB_J} c_b \cdot b \mod  (\widetilde{D}_{n,J}, D_{n,1}, D_{n,2}, \dots, D_{n,n-1})\quad \text{for some } c_b \in \FF_q.
    \end{equation}
    The product $g \cdot \theta_J$ therefore satisfies 
    \begin{equation}
    \label{eq:straightening}
        g \cdot \theta_J \equiv \sum_{b \in \BBB_J} c_b \cdot b \cdot \theta_J + \mathbf{E} \mod SI_G
    \end{equation}
    where the `error term' $\mathbf{E}$ satisfies
    \begin{equation}
        \mathbf{E} \in \bigoplus_{J <_\Gale K} S \cdot \theta_K.
    \end{equation}
    Since the straightening relation~\eqref{eq:straightening} holds for any $J \subseteq [n]$, we conclude that 
    \begin{equation}
        \Hilb(SR_G;t,z) \leq \sum_{J \subseteq [n]} z^{|J|} \times \Hilb(S/(\widetilde{D}_{n,J}, D_{n,1}, D_{n,2}, \dots, D_{n,n-1});t)
    \end{equation}
    and an application of Lemma~\ref{lem:tilde-quotient} completes the proof.
\end{proof}

\section{An Operator Theorem for $I^\perp_G$}
\label{sec:Operator}

In order to bound $SR_G$ from below, we first study commutative inverse systems. To motivate the next result, we recall a theorem of Steinberg \cite{SteinbergDifferential} which states that the characteristic 0 inverse system $I_{\symm_n}^\perp$ associated to the $\symm_n$-coinvariant ideal $I_{\symm_n}\subseteq S$ is generated by the Vandermonde determinant \[\delta = \sum_{w \in \symm_n} \sign(w) \cdot x_{w(1)}^{n-1} \cdots x_{w(n-1)}^1 x_{w(n)}^0.\]
(In characteristic 0, we identify the divided power ring element $y_i^{(a)}$ with $\frac{x_i}{a!}$.) We define a modular analog $\delta^{(q)} \in \DD$ of the Vandermonde determinant and prove a $GL_n(\FF_q)$-version of Steinberg's theorem.

\begin{proposition}
    \label{prop:steinberg}
    Let $G = GL_n(\FF_q)$ and consider the element $\delta^{(q)} \in \DD$ given by
    \begin{equation}
        \delta^{(q)} := \sum_{w \in \symm_n} \sign(w) \cdot y_{w(1)}^{(q^n-1-1)} y_{w(2)}^{(q^n-q-1)} \cdots y_{w(n)}^{(q^n - q^{n-1}-1)}.
    \end{equation}
    Then $\delta^{(q)}$ generates the inverse system $I_G^\perp \subseteq \DD$ under the $\odot$-action of $S$.
\end{proposition}

Adding 1 to exponents  appearing in $\delta^{(q)}$ yields the degrees $q^n-1, q^n-q,\dots,q^n -q^{n-1}$ of the fundamental invariants of $GL_n(\FF_q)$, just as in the case of $\symm_n$.

\begin{proof}
    It is not difficult to find $\dim_{\FF_q}(S/I_G) = |G| = \prod_{i=0}^{n-1} (q^n - q^i)$ linearly independent elements in $S \odot \delta^{(q)}$ by considering lexicographical leading terms and applying appropriate operators of the form $(x_1^{b_1} \cdots x_n^{b_n}) \odot (-)$. So the proposition reduces to proving that $\delta^{(q)} \in I_G^\perp$, i.e.
\begin{equation}
    D_{n,i} \odot \delta^{(q)} = 0 \quad \text{for $i = 0,1,\dots,n-1.$}
\end{equation}

    Let $\varepsilon_n := \sum_{w \in \symm_n} \sign(w) \cdot w \in \FF_q[\symm_n]$ be the antisymmetrizing element of the symmetric group algebra. In particular, we have
    \begin{equation}
        \delta^{(q)} = \varepsilon_n \cdot (y_1^{(q^n-1-1)} y_2^{(q^n - q - 1)} \cdots y_n^{(q^n - q^{n-1} - 1)}).
    \end{equation}
    Let $p := \sum_{i=1}^n (q^n - q^{i-1} - 1)$ be the sum of the exponents appearing in $\delta^{(q)}$. Define an $\FF_q$-linear operator $\sigma: S_p \to \FF_q$ by
    \begin{equation}
        \sigma(f) := \text{coefficient of $x_1^{q^n - 1-1} x_2^{q^n - q - 1} \cdots x_n^{q^n - q^{n-1}-1}$ in } \varepsilon_n \cdot f.
    \end{equation}
    for $f \in S_p$. It is not too hard to see that
    \begin{equation}
        f \odot \delta^{(q)} = \sigma(f) \quad \quad \text{for all $f \in S_p$.}
    \end{equation}
    Let $0 \leq i  \leq n-1$. Since $\odot$ induces a perfect pairing $\langle-,-\rangle : S_j \otimes \DD_j \to \FF_q$ for all $j$ and $\deg D_{n,i} = q^n - q^i$, in order to prove $D_{n,i} \odot \delta_n^{(q)} = 0$ it is enough to show
    \begin{equation}
        \label{eq:desired-sigma-equation}
        \sigma(D_{n,i} \cdot x_1^{b_1} \cdots x_n^{b_n} ) = 0 \quad \text{whenever} \quad
        b_1 + \cdots + b_n + q^n  -q^i = p.
    \end{equation}
    Starting with the left-hand side of  Equation~\eqref{eq:desired-sigma-equation}, we have
    \begin{align}
        \sigma(D_{n,i} \cdot x_1^{b_1} \cdots x_n^{b_n} )  &= 
        \text{coefficient of $x_1^{q^n - 1-1} x_2^{q^n - q - 1} \cdots x_n^{q^n - q^{n-1}-1}$ in } \varepsilon_n \cdot (D_{n,i} \cdot x_1^{b_1} \cdots x_n^{b_n}) \\
        &= \text{coefficient of $x_1^{q^n - 1-1} x_2^{q^n - q - 1} \cdots x_n^{q^n - q^{n-1}-1}$ in } D_{n,i} \times \varepsilon_n \cdot (x_1^{b_1} \cdots x_n^{b_n}) \\
        &= \text{constant term of } D_{n,i} \cdot x_1^{b_1} \cdots x_n^{b_n} \times \varepsilon_n \cdot (x_1^{-q^n +1+1} x_2^{-q^n + q + 1} \cdots x_n^{-q^n + q^{n-1} + 1})
    \end{align}
    where the second and third equalities used the fact that $D_{n,i}$ is $\symm_n$-invariant. The alternating expression $\varepsilon_n \cdot (x_1^{-q^n +1+1} x_2^{-q^n + q + 1} \cdots x_n^{-q^n + q^{n-1} + 1})$ on the previous line satisfies 
    \begin{equation}
        \varepsilon_n \cdot (x_1^{-q^n +1+1} x_2^{-q^n + q + 1} \cdots x_n^{-q^n + q^{n-1} + 1}) \doteq (x_1 \cdots x_n)^{1 - q^n} \times \det M_{\rho}
    \end{equation}
    where $\det M_{\rho}$ is the Moore determinant giving the denominator of $D_{n,i}$. We conclude that 
    \begin{multline}
    \label{eq:zero-constant-term}
        \sigma(D_{n,i} \cdot x_1^{b_1} \cdots x_n^{b_n} ) = \\ \pm \text{constant term of } 
        \det(M_{(n,n-1,\dots,i+1,i-1,\dots,1,0)}) \cdot x_1^{b_1 - q^n + 1} \cdots x_n^{b_n - q^n + 1}.
    \end{multline}
    However, since $0 \leq i \leq n-1$, the Moore matrix $M_{(n,n-1,\dots,i+1,i-1,\dots,1,0)}$ has first column entries $x_1^{q^n}, \dots, x_n^{q^n}$. As a consequence, every monomial $x_1^{c_1} \cdots x_n^{c_n}$ appearing in the expansion of $\det M_{(n,n-1,\dots,i+1,i-1,\dots,1,0)} $ satisfies $c_j = q^n$ for some $j$. In particular, since $b_j \geq 0$ for all $j$ we do {\bf not} have 
    $x_1^{c_1} \cdots x_n^{c_n} \times x_1^{b_1 - q^n + 1} \cdots x_n^{b_n - q^n + 1} = 1$. The constant term in \eqref{eq:zero-constant-term} therefore vanishes.
\end{proof}

Proposition~\ref{prop:steinberg} easily implies another result of Steinberg \cite{SteinbergBasis} on a basis for the polynomial $GL_n(\FF_q)$-coinvariant ring.

\begin{corollary}
    \label{cor:steinberg-basis}
    {\em (Steinberg \cite{SteinbergBasis})} Let $G = GL_n(\FF_q)$. The set of monomials
    \[
    \left\{
        x_1^{a_1} \cdots x_n^{a_n} \,:\, a_i < q^n - q^{n-i} \
    \right\}
    \]
    descends to a vector space basis of the coinvariant ring $R_G$.
\end{corollary} 

\begin{proof}
    By Proposition~\ref{prop:steinberg}, applying the given monomials $x_1^{a_1} \cdots x_n^{a_n}$ to $\delta^{(q)}$ under the $\odot$-action yields $\prod_{i=0}^{n-1} (q^n - q^i) = |G|$ elements of $I_G^\perp$ with distinct lexicographical leading terms. It follows that the given set of monomials descends to a linearly independent subset of $R_G$ of size $|G|$. Since $\dim_{\FF_q} R_G = |G|$, we are done.
\end{proof}

\section{Hilbert Series and Operator Theorem}
\label{sec:Hilbert}

Let $G = GL_n(\FF_q)$. In this section we prove our main results by calculating the Hilbert series of $SR_G$ and giving an operator-theoretic description of its inverse system; the proofs of these results will be intertwined. Before proceeding further, we give a precise description of the inverse system $I^\perp$ associated to a bigraded ideal $I \subseteq \Omega$, including the vector space in which $I^\perp$ lives.

\subsection{Superspace inverse systems}
For the current subsection, we work over an arbitrary field $\FF$.
Recall our notation $E = \wedge \{ \theta_1, \dots, \theta_n\}$ for the free rank $n$ exterior algebra over $\FF$. 
We have the {\em exterior differentiation} or {\em contraction} module structure
\begin{equation}
    \odot: E \otimes E \to E
\end{equation}
characterized by 
\begin{equation}
    \theta_i \odot \theta_J :=  \begin{cases}
        (-1)^{s-1} \theta_{J-i} & \text{ if $j_s = i$,} \\
        0 &\text{otherwise,}
    \end{cases}
\end{equation}
for any subset $J = \{j_1 < \cdots < j_k\} \subseteq [n]$. 
For any exterior degree $j$ we have a perfect pairing
\begin{equation}
    \langle -, - \rangle : E_j \otimes E_j \to \FF \quad \quad \langle f,g\rangle := f \odot g.
\end{equation}

In Section~\ref{sec:Background} we defined an action $\odot: S \otimes \DD \to \DD$ inducing a  perfect pairing $\langle -, - \rangle: S_i \otimes \DD_i \to \FF$ for all commutative degrees $i$. Making the identification $\Omega = S \otimes E$, we get an action
\begin{equation}
    \odot : \Omega \times (\DD \otimes  E) \to \DD \otimes  E \quad \quad (f_1 \otimes f_2) \odot (g_1 \otimes g_2) := (f_1 \odot g_1) \otimes (f_2 \odot g_2).
\end{equation}
For any fixed bidegree $(i,j)$, we have a perfect pairing
\begin{equation}
    \langle-,-\rangle: \Omega_{i,j} \otimes (\DD_i \otimes E_j) \to \FF
\end{equation}
given as before by $\langle f, g \rangle := f \odot g$.

Let $I \subseteq \Omega$ be a bihomogeneous ideal. The {\em inverse system} $I^\perp \subseteq \DD \otimes_\FF E$ is 
\begin{equation}
    I^\perp := \{ g \in \DD \otimes E \,:\, f \odot g = 0 \text{ for all } f \in I \}.
\end{equation}
It is not difficult to check that $I^\perp$ is a bigraded $\FF$-subspace of $\DD \otimes E$ which is closed under the $\odot$-action. Furthermore, since $\langle-,-\rangle$ is a perfect pairing on any bidegree, one has 
\begin{equation}
    \Hilb(\Omega/I;t,z) = \Hilb(I^\perp;t,z).
\end{equation}

\subsection{Modular higher Euler operators} We return to the setting of $\FF_q$.
For $j \geq 0$, define the {\em higher $\FF_q$-Euler operator} 
\begin{equation}
    d_j^{(q)}: \DD \otimes E \to \DD \otimes E 
\end{equation}
by the equation
\begin{equation}
    d_j^{(q)}(f) := \sum_{i=1}^n (x_i^{q^j} \odot f) \wedge \theta_i.
\end{equation}
Thus $d_j^{(q)}$ lowers commutative degree by $q^j$ while raising exterior degree by $1$. We prove that the operators $d_j^{(q)}$ preserve the superspace inverse system $SI_G^\perp \subseteq \DD \otimes E$ associated to $G = GL_n(\FF_q)$. In fact, this preservation holds for any subgroup $G \subseteq GL_n(\FF_q)$.

\begin{lemma}
    \label{lem:modular-euler-preservation}
    Let $G \subseteq GL_n(\FF_q)$ be any subgroup and let $SI_G = (\Omega^G_+) \subseteq \Omega$. For all $j \geq 0$, the operator $d_j^{(q)}$ preserves $SI_G^\perp$.
\end{lemma}

\begin{proof}
 We know that $SI_G^\perp$ is a bigraded subspace of $\DD \otimes E$. Let $f \in SI_G^\perp$ be a bihomogeneous element. We need to prove
\begin{equation}
    h \odot d_j^{(q)}(f) = 0 \quad \text{for all $h \in SI_G$.}
\end{equation}
Since $\odot$ restricts to a perfect pairing $\Omega_{a,b} \otimes (\DD_a \otimes E_b) \to \FF_q$ for each bidegree $(a,b)$, it does not hurt to assume that $h \in SI_G$ is bihomogeneous of the same bidegree as $d_j^{(q)}(f)$ so that $h \odot d_j^{(q)}(f) \in \FF_q$ is a scalar. This scalar is a nonzero multiple of
\begin{equation}
\label{eq:want-to-be-zero}
    h \odot d_j^{(q)}(f)  \doteq \left( \sum_{i=1}^n  x_i^{q^j} \cdot (\theta_i \odot h) \right) \odot f \in \FF_q.
\end{equation}
With \eqref{eq:want-to-be-zero} as motivation, consider the superspace operator $\widetilde{d}^{(q)}_j: \Omega \to \Omega$ given by
\begin{equation}
    \widetilde{d}^{(q)}_j(h) := \left( \sum_{i=1}^n  x_i^{q^j} \cdot (\theta_i \odot h) \right).
\end{equation}
We make two claims about the operator $\widetilde{d}^{(q)}_j$.
\begin{enumerate}
    \item If $h \in \Omega^G$ is a $G$-invariant superspace element, so is $\widetilde{d}_j^{(q)}(h)$.
    \item The operator $\widetilde{d}^{(q)}_j$ preserves the ideal $SI_G$.
\end{enumerate}

To see why assertion (1) is true, start by observing that the $\FF_q$-linear map $\psi_j: V^* \to S$ characterized by $\psi_j: \theta_i \mapsto x_i^{q^j}$ is $G$-equivariant.\footnote{It is enough to show $\psi_j$ is $G$-equivariant for $G = GL_n(\FF_q)$. Recall that $ GL_n(\FF_q)$ is generated by elementary matrices corresponding to switching rows, scaling a single row, and adding an $\FF_q$-multiple of one row to another row. Since the Frobenius map $x \mapsto x^q$ is a ring map fixing $\FF_q$, it follows that $G$ acts on $\{\theta_1,\dots,\theta_n\}$ in the same way as $\{ x_1^{q^j}, \dots, x_n^{q^j}\}$.} 
The operator $\widetilde{d}^{(q)}_j$ acts on the superspace ring
\[\Omega = S \otimes E = S \otimes \wedge (V^*) \] 
by the contraction rule
\begin{equation}
    \widetilde{d}^{(q)}_j: 1 \otimes \lambda_1 \wedge \cdots \wedge \lambda_\ell \mapsto \sum_{r=1}^\ell (-1)^{r-1} \psi_j(\lambda_r) \otimes \lambda_1 \wedge \cdots \wedge \widehat{\lambda_r} \wedge \cdots \wedge \lambda_\ell
\end{equation}
for $\lambda_r \in V^*$ and $S$-linearity. For any $f \in S$ and $g \in G$ one has
\begin{align*}
    \widetilde{d}^{(q)}_j( g \cdot (f \otimes \lambda_1 \wedge \cdots \wedge \lambda_\ell)) &= 
    \widetilde{d}^{(q)}_j( (g \cdot f) \otimes (g \cdot \lambda_1 \wedge \cdots \wedge g \cdot \lambda_\ell)) \\ &= 
    \sum_{r=1}^\ell (-1)^{r-1} \left[ (g \cdot f) \psi_j(g \cdot \lambda_r) \right] \otimes (g \cdot \lambda_1 \wedge \cdots \wedge \widehat{g \cdot \lambda_r} \wedge \cdots \wedge g \cdot \lambda_\ell) \\
    &= \sum_{r=1}^\ell (-1)^{r-1} 
    \left[ (g \cdot f) (g \cdot \psi_j(\lambda_r)) \right] \otimes (g \cdot \lambda_1 \wedge \cdots \wedge \widehat{g \cdot \lambda_r} \wedge \cdots \wedge g \cdot \lambda_\ell) \\
    &= g \cdot \sum_{r=1}^\ell (-1)^{r-1} 
    (f  \psi_j(\lambda_r))  \otimes (\lambda_1 \wedge \cdots \wedge \widehat{\lambda_r} \wedge \cdots \wedge \lambda_\ell)  \\
    &= g \cdot \widetilde{d}^{(q)}_j(f \otimes \lambda_1 \wedge \cdots \wedge \lambda_\ell)
\end{align*}
so  that $\widetilde{d}_j^{(q)}$ inherits $G$-equivariance from $\psi_j$ and assertion (1) follows. 

For assertion (2), use the fact that $\widetilde{d}_j^{(q)}$ satisfies the following {\em superspace Leibniz rule}: for any bihomogeneous $f_1, f_2 \in \Omega$ one has 
\begin{equation}
    \widetilde{d}^{(q)}_j(f_1 \wedge f_2) = f_1 \wedge \widetilde{d}^{(q)}_j(f_2) \pm
    d^{(q)}_j(f_1) \wedge f_2
\end{equation}
where the sign $\pm$ depends on the fermionic degree of $f_2$. Taken in conjunction with assertion (1), this Leibniz property implies that $\widetilde{d}^{(q)}_G$ preserves the ideal $SI_G \subseteq \Omega$ generated by $\Omega^G_+$.

We return to the scalar \eqref{eq:want-to-be-zero} given by
\begin{equation}
     h \odot d_j^{(q)}(f)  \doteq \left( \sum_{i=1}^n  x_i^{q^j} \cdot (\theta_i \odot h) \right) \odot f = \widetilde{d}^{(q)}_j(h) \odot f.
\end{equation}
If $h \in SI_G$ and $f \in SI_G^\perp$, assertion (2) implies $\widetilde{d}^{(q)}_j(h) \in SI_G$ so that $\widetilde{d}^{(q)}_j(h) \odot f = 0$. We conclude that $d_j^{(q)}$ preserves $SI_G^\perp$.
\end{proof}

\subsection{$\DDD^{(q)}_J$-operators}
We use the operators $d^{(q)}_j$ to build more elaborate operators $\DDD^{(q)}_J$ on $\DD \otimes_{\FF_q} E$ which preserve the inverse system $SI_G^\perp$. Our methods should be compared to those for the characteristic 0 case of $\symm_n$ appearing in \cite[Sec. 4]{MRW}. We begin by defining two matrices with polynomial entries.

Let $\zz = (z_1,\dots,z_k)$ be a list of $k$ variables. The $k \times n$ {\em $\FF_q$-power matrix} $M_{n,k}(\zz)$ is the matrix with entries
\begin{equation}
    M_{n,k}(\zz)_{i,j} := z_i^{q^{n-j}}.
\end{equation}
For example, if $n = 5$ and $k = 3$ we have
\[
M_{5,3}(\zz) = M_{5,3}(z_1,z_2,z_3) = \begin{pmatrix}
    z_1^{q^4} & z_1^{q^3} & z_1^{q^2} & z_1^{q^1} & z_1^{q^0} \\
    z_2^{q^4} & z_2^{q^3} & z_2^{q^2} & z_2^{q^1} & z_2^{q^0} \\
    z_3^{q^4} & z_3^{q^3} & z_3^{q^2} & z_3^{q^1} & z_3^{q^0}
\end{pmatrix}.
\]
Recall from Definition~\ref{def:P-Q} that $P_i(t) \in S[t]$ is the degree $q^{n-i+1}$ polynomial
\[
P_i(t) = \prod_{a_i,\dots,a_n \in \FF_q} (t + a_i x_i + \cdots + a_n x_n).
\]
We define the $k \times n$ {\em $q$-factor matrix} $F_{n,k}(\xx,\zz)$ by
\begin{equation}
    F_{n,k}(\xx,\zz)_{i,j} := P_{j+1}(z_i)
\end{equation}
where we recall $P_{n+1}(t) = t$.
When $k = 3$ and $n = 4$ we have
\[
F_{5,3}(\xx,\zz) = \begin{pmatrix}
    P_2(z_1) & P_3(z_1) & P_4(z_1) & P_5(z_1) & P_6(z_1) \\ 
    P_2(z_2) & P_3(z_2) & P_4(z_2) & P_5(z_2) & P_6(z_2) \\ 
    P_2(z_3) & P_3(z_3) & P_4(z_3) & P_5(z_3) & P_6(z_3) \\ 
\end{pmatrix}
\]
where $P_6(t) = t$.

\begin{lemma}
    \label{lem:unitriangular}
    There exists an $n \times n$ lower unitriangular matrix $C$ with entries in $S$ such that 
    \[
    M_{n,k}(\zz)\cdot C = F_{n,k}(\xx,\zz).
    \]
\end{lemma}

\begin{proof}
    Lemma~\ref{lem:simple-PQ} (2) implies that the polynomial $P_i(t)$ expands in terms of Dickson invariants in partial variable sets according to 
    \begin{equation}
    P_i(t) = t^{q^{n-i+1}} + \sum_{r = 0}^{n-i} D_{n-i+1,r}(x_i, x_{i+1}, \dots, x_n) \cdot t^{q^r}.
    \end{equation}
    Thanks to this expansion, for $1 \leq j < i \leq n$ we can take the $(i,j)$-entry of $C$ to be the Dickson invariant $D_{n-j+1,n-i+1}(x_j,x_{j+1},\dots,x_n)$.
\end{proof}

Given a $k$-element subset $J \subseteq [n]$, let $E_J$ be the $(n-k) \times n$ echelon matrix with 0,1-entries whose pivot columns are indexed by $[n] - J$. For example, if $n  =5$ and $J = \{1,4,5\}$ then 
\[
E_J = \begin{pmatrix}
    0 & 1 & 0 & 0 & 0 \\
    0 & 0 & 1 & 0 & 0 
\end{pmatrix}.
\]
Define a $(n-k) \times k$ matrix $H_J$ with entries in $S$ by
\begin{equation}
    H_J := E_J \cdot C^{-1}
\end{equation}
where $C$ is the $n \times n$ matrix of Lemma~\ref{lem:unitriangular}. Let $\widetilde{F}_{n,k,J}(\xx,\zz)$ be the $n \times n$ {\em augmented factor matrix} given in block form by
\begin{equation}
    \widetilde{F}_{n,k,J}(\xx,\zz) := \left(\begin{array}{c}
            F_{n,k}(\xx,\zz) \\ \hline E_J
    \end{array}\right)
\end{equation}  
Observe that 
\begin{equation}
    \widetilde{F}_{n,k,J}(\xx,\zz) \cdot C^{-1} = 
    \left(\begin{array}{c}
            F_{n,k}(\xx,\zz) \cdot C^{-1} \\ \hline E_J \cdot C^{-1} 
    \end{array}\right) = 
    \left(\begin{array}{c}
            M_{n,k}(\zz) \\ \hline H_J
    \end{array}\right).
\end{equation}
For any subset $J = \{ j_1 < \cdots < j_k \} \subseteq \{0,1,\dots,n-1\}$, define an operator
\begin{equation}
    d^{(q)}_J: \DD \otimes E \longrightarrow \DD \otimes E
\end{equation}
by the composition $d^{(q)}_J := d_{j_1}^{(q)} \circ \cdots \circ d^{(q)}_{j_k}$. We have all the tools we need to define the $\DDD$-operators.

\begin{defn}
    \label{def:D-operator}
    Let $J \subseteq [n]$ satisfy $\# J = k$. We define an operator $\DDD_J^{(q)}$ on $\DD \otimes E$ by the formula
    \begin{equation}
        \DDD_J^{(q)}(f) := \sum_{\substack{U \subseteq [n] \\ \#U = n-k}} (-1)^{\sum U} \Delta_U(H_J) \odot d^{(q)}_{U^*}(f)
    \end{equation}
    where $\sum U := \sum_{u \in U} u$, $\Delta_U(H_J)$ is the maximal minor of $H_J$ with columns indexed by $U$, and $U^* := \{n-u \,:\, u \in U \} \subseteq \{0,1,\dots,n-1\}$.
\end{defn}

Definition~\ref{def:D-operator} gives an $\FF_q$-analog of the $\DDD$-operators appearing in \cite{RW} and \cite{MRW}. As in those papers, Definition~\ref{def:D-operator} is designed to give the following triangularity result. Recall that for $1 \leq j \leq n$, the polynomial $Q_j$ is given by
\[
Q_j = \prod_{a_{j+1},\dots,a_n \in \FF_q} (x_j + a_{j+1} x_{j+1} + \cdots + a_n x_n) \in S
\]
and for $J \subseteq [n]$ we have $Q_J = \prod_{j \in J} Q_j$.

\begin{lemma}
    \label{lem:D-triangular}
    Let $J,K \subseteq [n]$ with $\# J = \# K = [n]$ and let $f \in \DD$. Consider the element \[\DDD_J^{(q)}(f) \in \DD \otimes_{\FF_q} E.\]
    The coefficient of $\theta_K$ in $\DDD_J^{(q)}(f)$ lies in $\DD$.
    \begin{enumerate}
        \item If $K \not\leq_\Gale J$, the coefficient of $\theta_K$ in $\DDD_J^{(q)}(f)$ vanishes.
        \item If $K = J$, the coefficient of $\theta_J$ in $\DDD_J^{(q)}(f)$ is $Q_J \odot f$, up to a nonzero scalar.
    \end{enumerate}
\end{lemma}

\begin{proof}
Fix $K \subseteq [n]$ with $|K| = k$. We write $M_{n,k}(\xx_K)$ and $F_{n,k}(\xx,\xx_K)$ for the matrices obtained from $M_{n,k}(\zz)$ and $F_{n,k}(\xx,\zz)$ by replacing the variables in the $k$-element set $\zz = \{z_1,\dots,z_k\}$ with those in the $k$-element set $\xx_K = \{x_i \,:\, i \in K \}$ in some fixed order. The choice of this order will only produce immaterial signs.

For any $f \in S$ we compute
\begin{align}
    \text{coefficient of $\theta_K$ in } \DDD^{(q)}_J(f) &=
    \text{coefficient of $\theta_K$ in } \sum_{\substack{U \subseteq [n] \\ \#U = n-k}} (-1)^{\sum U} \Delta_U(H_J) \odot d^{(q)}_{U^*}(f) \\
    &\doteq \det \left( 
        \begin{array}{c}
            M_{n,k}(\xx_K) \\ \hline H_J
        \end{array}
    \right) \odot f \\
    &= \det \left( \begin{array}{c}
            F_{n,k}(\xx,\xx_K) \\ \hline E_J
        \end{array}
    \right) \odot f
\end{align}
where the last equality holds because the transition matrix $C$ in Lemma~\ref{lem:unitriangular} is unitriangular.  Write $\overline{F}_{n,J}(\xx,\xx_K)$ for the `reduced' factor matrix of size $k \times k$ obtained by selecting the columns in $F_{n,k}(\xx,\xx_K)$ indexed by $J \subseteq [n]$.  We have
\begin{equation}
    \text{coefficient of $\theta_K$ in } \DDD^{(q)}_J(f) \doteq 
    \det \left( \begin{array}{c}
            F_{n,k}(\xx,\xx_K) \\ \hline E_J
        \end{array}
    \right) \odot f \doteq 
    \det \overline{F}_{n,J}(\xx,\xx_K) \odot f.
\end{equation}

We examine the determinant $\det \overline{F}_{n,J}(\xx,\xx_K)$ in detail.
Since $P_i(x_j) = \left[ P_i(t) \right]_{t = x_j}$ vanishes whenever $j \geq i$ (Lemma~\ref{lem:simple-PQ} (3)), whenever $K \not\leq_\Gale J$ the determinant of $\overline{F}_{n,J}(\xx,\xx_K)$ has block form
\begin{equation}
    \det \overline{F}_{n,J}(\xx,\xx_K) = \det \begin{pmatrix}
        * & * \\ \zero & * 
    \end{pmatrix}
\end{equation}
where the southwest block of zeros meets the main diagonal. This forces $\det \overline{F}_{n,J}(\xx,\xx_K) = 0$ and (1) follows.

On the other hand, when $K = J$ the matrix $\overline{F}_{n,J}(\xx,\xx_J)$ is upper triangular with diagonal elements
\[
 \{ P_{j+1}(x_j)  \,:\, j \in J \}
\]
so that
\begin{equation}
    \det \overline{F}_{n,J}(\xx,\xx_J) = \prod_{j \in J} P_{j+1}(x_j) = Q_J
\end{equation}
which yields (2).
\end{proof}

\subsection{Main results} We have all of the necessary lemmata to give our promised formula for the Hilbert series of the $GL_n(\FF_q)$-superspace coinvariant ring.

\begin{theorem}
    \label{thm:hilbert}
    The Hilbert series of the $G = GL_n(\FF_q)$-superspace coinvariant ring is given by
    \begin{equation}
        \Hilb(SR_G;t,z) = \prod_{i=1}^{n-1} [q^n - q^i]_t \times \sum_{J \subseteq [n]} z^{|J|} \cdot  \left[ q^n - 1 - \sum_{j \in J} q^{n-j} \right]_t.
    \end{equation}
\end{theorem}

\begin{proof}
    Lemma~\ref{lem:hilbert-upper-bound} gives the upper bound
    \begin{equation}
        \Hilb(SR_G;t,z) \leq \prod_{i=1}^{n-1} [q^n - q^i]_t \times \sum_{J \subseteq [n]} z^{|J|} \cdot  \left[ q^n - 1 - \sum_{j \in J} q^{n-j} \right]_t.
    \end{equation}
    Since $\Hilb(SR_G;t,z) = \Hilb(SI_G^\perp;t,z)$, it suffices to prove 
    \begin{equation}
    \label{eq:desired-inverse-lower-bound}
        \Hilb(SI_G^\perp;t,z) \geq \prod_{i=1}^{n-1} [q^n - q^i]_t \times \sum_{J \subseteq [n]} z^{|J|} \cdot  \left[ q^n - 1 - \sum_{j \in J} q^{n-j} \right]_t.
    \end{equation}
    We establish the inequality \eqref{eq:desired-inverse-lower-bound} as follows.

    Let $\delta^{(q)} \in \DD$ be the divided power ring element appearing in Proposition~\ref{prop:steinberg}. Fix $J \subseteq [n]$. Proposition~\ref{prop:steinberg}, Lemma~\ref{lem:modular-euler-preservation} and the explicit formula for the operator $\DDD_J^{(q)}$ in Definition~\ref{def:D-operator} imply that
    \begin{equation}
        \DDD^{(q)}_J(\delta^{(q)}) \in SI_G^\perp.
    \end{equation}
    Lemma~\ref{lem:D-triangular} implies that 
    \begin{equation}
    \label{eq:triangular-consequence}
        \DDD^{(q)}_J(\delta^{(q)}) = (Q_J\odot \delta^{(q)}) \cdot \theta_J + \text{a $\DD$-linear combination of $\theta_K$ for $K <_\Gale J.$} 
    \end{equation}
    For any $\gamma \in \DD$, let $\ann_S(\gamma) = \{ f \in S \,:\, f \odot \gamma = 0 \}$ be the annihilator of $\gamma$ in $S$. Proposition~\ref{prop:steinberg} implies
    \begin{equation}
        \label{eq:annihilator-identification}
        \ann_S (  Q_J  \odot \delta^{(q)} ) = 
        (\ann_S(\delta^{(q)}) : Q_J)
        = 
        \left(
            I_G: Q_J
        \right)
    \end{equation}
    as ideals in $S$. Since $SI_G^\perp$ is closed under the $\odot$-action of $S$ and the triangularity property \eqref{eq:triangular-consequence} holds for each $J \subseteq [n]$, we have 
    \begin{align}
        \Hilb(SI_G^\perp;t,z) &\geq \Hilb \left(
        \sum_{J \subseteq [n]} S \odot \DDD^{(q)}_J(\delta^{(q)});t,z
        \right)\\
        &\geq \sum_{J \subseteq [n]} z^{|J|} \cdot \Hilb \left(
            S \odot \left( Q_J \odot \delta^{(q)} \right); t
        \right) \\
        &= \sum_{J \subseteq [n]} z^{|J|} \cdot \Hilb\left(
            \frac{S}{\ann_S\left( Q_J \odot \delta^{(q)}  \right)}; t \right) \\
        &= \sum_{J \subseteq [n]} z^{|J|}
            \cdot \Hilb \left( S/(I_G:Q_J); t \right) \\
        &= \sum_{J \subseteq [n]} z^{|J|} \cdot \prod_{i=1}^{n-1} [q^n - q^i]_t \times \left[ q^n - 1 - \sum_{j \in J} q^{n-j} \right]_t
    \end{align}
    where the last equality follows from Lemma~\ref{lem:tilde-quotient}. This establishes the desired lower bound \eqref{eq:desired-inverse-lower-bound} on $\Hilb(SI_G^\perp;t,z)$ and the proof is complete. 
\end{proof}

The proof of Theorem~\ref{thm:hilbert} also establishes an operator-theoretic characterization of the inverse system $SI_G^\perp$.

\begin{theorem}
    \label{thm:operator}
    The inverse system $SI_G^\perp \subseteq \DD \otimes_{\FF_q} E$ associated to $G = GL_n(\FF_q)$ is the smallest linear subspace of $\DD \otimes E$ which $\dots$
    \begin{itemize}
        \item contains the element $\delta^{(q)}$,
        \item is closed under the $\odot$-action of $S$, and
        \item is closed under the $\FF_q$-higher Euler operators $d_j^{(q)}$ for $0 \leq j \leq n-1$.
    \end{itemize}
\end{theorem}

\begin{proof}
    Let $SH_G \subseteq \DD \otimes_{\FF_q} E$ be the linear subspace described in the theorem. By Proposition~\ref{prop:steinberg} and Lemma~\ref{lem:modular-euler-preservation} we have
    \begin{equation}
        SH_G \subseteq SI_G^\perp.
    \end{equation}
    Definition~\ref{def:D-operator} makes it apparent that 
    \begin{equation}
        \sum_{J \subseteq [n]} S \odot \DDD^{(q)}_J (\delta^{(q)}) \subseteq SH_G
    \end{equation}
    as bigraded subspaces of $\DD \otimes_{\FF_q} E$.
    On the other hand, we have
    \begin{align}
        \dim_{\FF_q} \left( \sum_{J \subseteq [n]} S \odot \DDD^{(q)}_J (\delta^{(q)}) \right) &\geq \prod_{i=1}^{n-1}(q^n - q^i) \times \sum_{J \subseteq [n]} \left(q^n - 1 - \sum_{j \in J} q^{n-j}\right) \\
        &= \dim_{\FF_q} SI_G^\perp
    \end{align}
    where the equality follows from Theorem~\ref{thm:hilbert} and the inequality follows from its proof. The nested vector spaces
    \[
    \sum_{J \subseteq [n]} S \odot \DDD_J^{(q)}(\delta^{(q)}) \subseteq SH_G \subseteq SI_G^\perp
    \]
    therefore satisfy
    \[
    \sum_{J = [n]} S \odot \DDD_J^{(q)}(\delta^{(q)}) \subseteq SH_G = SI_G^\perp
    \]
    and the proof is complete.
\end{proof}

\section{Groups Between $SL_n(\FF_q)$ and $GL_n(\FF_q)$}
\label{sec:Between}

In this section we generalize Theorems~\ref{thm:hilbert} and \ref{thm:operator} to subgroups $G$ of $GL_n(\FF_q)$ which contain $SL_n(\FF_q)$. The method of proof is quite similar; we indicate where to make the relevant adjustments. The replacement for Lemma~\ref{lem:syzygy} reads as follows.

\begin{lemma}
    \label{lem:syzygy-between}
    Let $G$ be a subgroup of $GL_n(\FF_q)$ containing $SL_n(\FF_q)$ and let $e = |\{\det(g)\,:\, g \in G \} |$.
    For any $1 \leq i \leq n$, there exists a bihomogeneous element $f_i^G \in SI_G$ belonging to the superspace coinvariant ideal attached to $G$ such that the following statements are true.
    \begin{itemize}
        \item $f_i^G$ has fermionic degree 1 and bosonic degree $e \cdot (q^{n-1} + \cdots + q + 1) - q^{n-i}.$
        \item The coefficient of $\theta_i$ in $f_i^G$ is 
        \[
        \frac{(\det M_\rho)^e}{Q_i}.
        \]
        \item The coefficient of $\theta_j$ in $f_i^G$ is zero for $i' < i$.
    \end{itemize}
\end{lemma}

\begin{proof}
    The Moore matrix determinant $\det M_\rho$ for $\rho = (n-1,\dots,1,0)$ satisfies
    \begin{equation}
        \label{eq:moore-into-Q}
        \det M_\rho \doteq Q_1 Q_2 \cdots Q_n.
    \end{equation}
     By Theorem~\ref{thm:HS} we may divide the element $f_i \in SI_{GL_n(\FF_q)}$ appearing in Lemma~\ref{lem:syzygy} by $(\det M_\rho)^{q-1-e}$ to obtain an element $f_i^G \in SI_G$ of the required form.
\end{proof}

Just as in Lemma~\ref{lem:hilbert-upper-bound}, we can use Lemma~\ref{lem:syzygy-between} to obtain the following upper bound on the Hilbert series of $SR_G$:
\begin{equation}
    \label{eq:between-hilbert-upper}
    \Hilb(SR_G;t,z) \leq \prod_{i=1}^{n-1} [q^n - q^i]_t \times \sum_{J \subseteq [n]} z^{|J|} \cdot \left[ e \cdot (q^{n-1} + \cdots + q + 1) - \sum_{j\in J} q^{n-j} \right]_t.
\end{equation}
The inequality \eqref{eq:between-hilbert-upper} is proven in exactly the same way as Lemma~\ref{lem:hilbert-upper-bound}, but with the $f_i$ of Lemma~\ref{lem:syzygy} replaced by the $f_i^G$ of Lemma~\ref{lem:syzygy-between}. 

The next step is to analyze the commutative inverse system $I_G^\perp \subseteq \DD$ for $G$ containing $SL_n(\FF_q)$. Define an element $\delta_G^{(q)} \in \DD$ by
\begin{align}
    \delta_G^{(q)} &:= (\det M_{\rho})^{q-1-e} \odot  \left[\sum_{w \in \symm_n} \sign(w) \cdot y_{w(1)}^{(q^n-1-1)} y_{w(2)}^{(q^n - q-1)} \cdots y_{w(n)}^{(q^n-q^{n-1}-1)} \right] \\
    &= (\det M_\rho)^{q-1-e} \odot \delta^{(q)}
\end{align}
where $e := | \{\det(g) \,:\, g \in G \}|$ and $\delta^{(q)} \in \DD$ is as in Proposition~\ref{prop:steinberg}. In the case $G = GL_n(\FF_q)$ we have $\delta_{GL_n(\FF_q)}^{(q)} = \delta^{(q)}$. The $G$-analog of 
Proposition~\ref{prop:steinberg} is as follows.

\begin{lemma}
    \label{lem:modular-steinberg-between}
    Let $G$ be a subgroup of $GL_n(\FF_q)$ containing $SL_n(\FF_q)$.
    The divided power ring element $\delta_G^{(q)} \in \DD$ generates $I_G^\perp$ under the $\odot$-action.
\end{lemma}

\begin{proof}
    Since $D_{n,0} \doteq (\det M_\delta)^{q-1}$, it follows from the definition of $\delta_G^{(q)}$ and Proposition~\ref{prop:steinberg} that $\delta^{(q)}_G \in SI_G^\perp$. We are therefore reduced to proving
    \begin{equation}
        \label{eq:desired-between-inverse-dimension-polynomial}
        \dim_{\FF_q}  S \odot \delta^{(q)}_G = \dim_{\FF_q} I_G^\perp = \dim_{\FF_q} R_G.
    \end{equation}
    By Proposition~\ref{prop:steinberg} and the definition of $\delta_G^{(q)}$ we have 
    \begin{equation}
        \label{eq:known-between-inverse-dimension-polynomial}
        \dim_{\FF_q}  S \odot \delta^{(q)}_G = \dim_{\FF_q} S/(I_{GL_n(\FF_q)} : (\det M_\rho)^{q-1-e}).
    \end{equation}
    Since $D_{n,0} \doteq (\det M_\rho)^{q-1}$ and $I_{GL_n(\FF_q)} = (D_{n,0}, D_{n,1}, \dots, D_{n,n-1})$, an application of Lemma~\ref{lem:colon-regular} gives 
    \begin{equation}
        (I_{GL_n(\FF_q)} : (\det M_\rho)^{q-1-e}) = ((\det M_\rho)^e, D_{n,1}, \dots, D_{n,n-1}) = I_G
    \end{equation}
    so that the dimensions \eqref{eq:desired-between-inverse-dimension-polynomial}  and \eqref{eq:known-between-inverse-dimension-polynomial} coincide and the proof is complete.
\end{proof}

Since Lemma~\ref{lem:modular-euler-preservation} holds for any subgroup $G$ containing $SL_n(\FF_q)$\footnote{actually for any subgroup $G$ of $GL_n(\FF_q)$}, we may apply the operators $\DDD^{(q)}_J$ of Definition~\ref{def:D-operator} to $\delta^{(q)}_G$ to obtain elements of $SI_G^\perp$. This allows us to turn the upper bound \eqref{eq:between-hilbert-upper} on $\Hilb(SR_G;t,z)$ into an equality.

\begin{theorem}
    \label{thm:hilbert-between}
    Let $G$ be a subgroup of $GL_n(\FF_q)$ containing $SL_n(\FF_q)$ and let $e := | \{ \det(g) \,:\, g \in G \} |$.
    We have
    \begin{equation}
        \Hilb(SR_G;t,z) = \prod_{i=1}^{n-1} [q^n - q^i]_t \times \sum_{J \subseteq [n]} z^{|J|} \cdot  \left[ e \cdot (q^{n-1} + \cdots + q + 1) - \sum_{j \in J} q^{n-j} \right]_t.
    \end{equation}
\end{theorem}

\begin{proof}
    Thanks to Equation~\eqref{eq:between-hilbert-upper} and the equality $\Hilb(SR_G;t,z) = \Hilb(SI^\perp_G;t,z)$ we need only show the lower bound
    \begin{equation}
    \label{eq:desired-between-inverse-lower-bound}
        \Hilb(SI_G^\perp;t,z) \geq \prod_{i=1}^{n-1} [q^n - q^i]_t \times \sum_{J \subseteq [n]} z^{|J|} \cdot  \left[ e \cdot (q^{n-1} + \cdots + q + 1) - \sum_{j \in J} q^{n-j} \right]_t.
    \end{equation}
    As explained in the paragraph before the statement of the theorem, for any $J \subseteq [n]$ we have an element
    \begin{equation}
        \DDD^{(q)}_J(\delta^{(q)}_G) \in SI_G^\perp.
    \end{equation}
    Lemma~\ref{lem:D-triangular} says that $\DDD^{(q)}_J(\delta^{(q)}_G)$ satisfies 
    \begin{equation}
    \label{eq:triangular-consequence-between}
        \DDD^{(q)}_J(\delta^{(q)}_G) = (Q_J \odot \delta^{(q)}_G) \cdot \theta_J + \text{a $\DD$-linear combination of $\theta_K$ for $K <_\Gale J.$} 
    \end{equation}
    Lemma~\ref{lem:modular-steinberg-between} implies 
    \begin{equation}
    \label{eq:annihilator-identification-between}
        \ann_S \left(
             Q_J \odot \delta^{(q)}_G 
        \right)  =
        \left(
                \ann_S(\delta_G^{(q)}): Q_J
        \right) = 
        \left(
                I_G :  Q_J
        \right)
    \end{equation}
    as ideals in $S$. Since $I_G = ( (\det M_\rho)^e, D_{n,1}, \dots, D_{n,n-1})$ where $\det M_\rho = Q_1 Q_2 \cdots Q_n$, the proof Lemma~\ref{lem:tilde-quotient} generalizes easily to show
    \begin{equation}
    \label{eq:colon-hilbert-between}
        \Hilb(S/(I_G:Q_J);t)
        = \prod_{i=1}^{n-1} [q^n - q^i]_t \times \left[ e \cdot (q^{n-1} + \cdots + q + 1) - \sum_{j \in J} q^{n-j}  \right]_t.
    \end{equation}
    Combining \eqref{eq:annihilator-identification-between} and \eqref{eq:colon-hilbert-between} with the triangularity statement \eqref{eq:triangular-consequence-between}, we get
    \begin{align}
        \Hilb(SI_G^\perp;t,z) &\geq  \Hilb \left(  \sum_{J \subseteq [n]} S \odot \DDD_J^{(q)}(\delta_G^{(q)});t,z \right) \\
        &\geq \sum_{J \subseteq [n]} z^{|J|} \cdot \Hilb\left(
            \frac{S}{\ann_S \left(
         Q_J \odot \delta^{(q)}_G 
        \right)};t
        \right) \\
        &= \sum_{J \subseteq [n]} z^{|J|} \cdot \Hilb\left(
            S/(I_G:Q_J);t
        \right) \\
        &= \prod_{i=1}^{n-1} [q^n - q^i]_t \times \sum_{J \subseteq [n]} z^{|J|} \cdot  \left[ e \cdot (q^{n-1} + \cdots + q + 1) - \sum_{j \in J} q^{n-j} \right]_t.
    \end{align}
    This establishes the desired lower bound \eqref{eq:desired-between-inverse-lower-bound} and the proof is complete.
\end{proof}

We close this section with an operator-theoretic characterization of the inverse system $SI_G^\perp$ for $G$ containing $SL_n(\FF_q)$. This result is exactly the same as the $GL_n(\FF_q)$-case of Theorem~\ref{thm:operator}, except that the inverse system generator $\delta^{(q)}$ is replaced by $\delta^{(q)}_G$.

\begin{theorem}
    \label{thm:operator-between}
    Let $G$ be a subgroup of $GL_n(\FF_q)$ containing $SL_n(\FF_q)$ and let $e := | \{ \det(g) \,:\, g \in G \} |$.
    The inverse system $SI_G^\perp \subseteq \DD \otimes E$ is the smallest linear subspace of $\DD \otimes E$ which $\dots$
    \begin{itemize}
        \item contains the element $\delta^{(q)}_G$,
        \item is closed under the $\odot$-action of $S$, and
        \item is closed under the $\FF_q$-higher Euler operators $d_i^{(q)}$ for $0 \leq i \leq n-1$.
    \end{itemize}
\end{theorem}

\begin{proof} 
    Let $SH_G \subseteq \DD \otimes E$ be the subspace in question. As in the proof of Theorem~\ref{thm:operator}, one has 
    \begin{equation}
        \sum_{J \subseteq [n]} S \odot \DDD^{(q)}_J(\delta^{(q)}_G) \subseteq SH_G \subseteq SI_G^\perp.
    \end{equation}
    The proof of Theorem~\ref{thm:hilbert-between} shows that the vector space dimensions of $\sum_{J \subseteq [n]} S \odot \DDD^{(q)}_J(\delta^{(q)}_G)$ and $SI_G^\perp$ coincide (and are equal to $|G|$) so that these subspace containments are actually equalities. 
\end{proof}

\section{Conclusion}
\label{sec:Conculsion}

Theorem~\ref{thm:hilbert-between} gives the Hilbert series of $SR_G$ for all subgroups $G$ of $GL_n(\FF_q)$ containing $SL_n(\FF_q)$. It is natural to ask for an explicit basis of this vector space.

\begin{problem}
    \label{prob:basis}
    Let $G$ be a subgroup of $GL_n(\FF_q)$ containing $SL_n(\FF_q)$. Find an explicit vector space basis of $SR_G$.
\end{problem}

In light of Theorem~\ref{thm:hilbert-between} and Steinberg's Corollary~\ref{cor:steinberg-basis}, a natural guess for a monomial basis of $SR_G$ is
\begin{equation}
    \AAA_G := \bigsqcup_{J \subseteq [n]} \AAA_{G,J} \cdot \theta_J
\end{equation}
where for $J \subseteq [n]$,
\begin{equation}
    \AAA_{G,J} = \left\{
        x_1^{a_1} \cdots x_n^{a_n} \,:\, \begin{array}{c}
            a_i < q^n - q^{n-i} \text{ for } i = 1,2,\dots,n-1 \text{ and } \\
            a_n < e \cdot (q^{n-1} + \cdots + q + 1) - \sum_{j \in J} q^{n-j} 
        \end{array}
    \right\}
\end{equation}
with $e = |\{ \det(g) \,:\, g \in G \}|$. It may be interesting to see if the methods used in \cite{ACKMR, BR} to find monomial bases for characteristic 0 superspace coinvariant rings for the complex reflection groups $G(r,1,n)$ can be applied to show that $\BBB_G$ descends to a basis of $SR_G$. A potentially helpful `supercommutative to commutative transfer principle' in this direction is as follows.

\begin{proposition}
    \label{prop:basis-transfer}
    Let $G$ be a subgroup of $GL_n(\FF_q)$ containing $SL_n(\FF_q)$. For any subset $J \subseteq [n]$, suppose we have a set $\BBB_{G,J} \subseteq S$ of homogeneous polynomials which descends to a vector space basis of $S/(I_G:Q_J)$. Then $\bigsqcup_{J \subseteq [n]} \BBB_{G,J} \cdot \theta_J$ descends to a vector space basis of $SR_G$. 
\end{proposition}

The proof of Proposition~\ref{prop:basis-transfer} is similar to that of the corresponding basis transfer result \cite[Thm. 5.4]{RW} in the characteristic 0 case of $\symm_n$ and is only sketched.

\begin{proof} (Sketch)
    An argument similar to the proof of Lemma~\ref{lem:hilbert-upper-bound} (but using Lemma~\ref{lem:syzygy-between} instead of Lemma~\ref{lem:syzygy}) shows that $\bigsqcup_{J \subseteq [n]} \BBB_{G,J} \cdot \theta_J$ descends to a spanning set of $SR_G$. The proof of Theorem~\ref{thm:hilbert-between} implies
    \begin{equation}
        \dim_{\FF_q} SR_G = \dim_{\FF_q} SI_G^\perp = \sum_{J \subseteq [n]} \dim_{\FF_q} S/(I_G:Q_J) = \sum_{J \subseteq [n]} |\BBB_{G,J}|
    \end{equation}
    so that the spanning set $\bigsqcup_{J \subseteq [n]} \BBB_{J,G} \cdot \theta_J$ of $SR_G$ is also a basis.
\end{proof}

Since $\dim_{\FF_q} S/(I_G:Q_J) = |\AAA_{G,J}|$ by the proof of Theorem~\ref{thm:hilbert-between}, in order to prove that $\AAA_G$ descends to a vector space basis of $SR_G$ it would suffice to show that $\AAA_{G,J}$ spans or is linearly independent in $S/(I_G:Q_J)$ for all $J \subseteq [n]$.

Another open problem concerns the module structure of $SR_G$. Since we are in modular characteristic, we state this problem in terms of Brauer-isomorphisms rather than module isomorphisms.

\begin{problem}
    \label{prob:module-structure}
    Let $G$ be a subgroup of $GL_n(\FF_q)$ containing $SL_n(\FF_q)$.
    Determine the Brauer-isomorphism class of $SR_G$ as $\dots$
    \begin{enumerate}
        \item an ungraded $G$-module,
        \item a singly-graded $G$-module with respect to exterior degree, and
        \item a bigraded $G$-module.
    \end{enumerate}
\end{problem}

As mentioned in the introduction, Mitchell proved \cite{Mitchell} that  the exterior degree 0 part of $SR_G$ (i.e. the polynomial $G$-coinvariant ring $R_G$) is Brauer-isomorphic to the regular representation $\FF_q[G]$. The graded Brauer-isomorphism type of $R_G$ is unknown in full generality, but Wan--Wang determined \cite{WW} many of its graded composition factors.

We do not have a conjectural solution to Problem~\ref{prob:module-structure} (2) in general, but in the case $G = SL_n(\FF_q)$ the dimension data of Theorem~\ref{thm:hilbert-between} is consistent with the conjecture that 
\begin{equation}
    \label{eq:sl-guess}
    \text{the exterior degree $j$ part of $SR_{SL_n(\FF_q)}$} \text{ is Brauer-isomorphic to } \FF_q[SL_n(\FF_q)]^{\oplus {n-1 \choose j}}.
\end{equation}

As discussed earlier, the Brauer-isomorphism \eqref{eq:sl-guess} holds when $j = 0$.
Let $\chi_j$ be the Brauer-character of the exterior degree $j$ part of $SR_{SL_n(\FF_q)}$ and suppose $q = p^a$ for $p$ prime. Thanks to Theorem~\ref{thm:hilbert-between}, the conjecture \eqref{eq:sl-guess} is equivalent to the assertion that $\chi_j$ vanishes on every nonidentity $p$-regular element $g \in SL_n(\FF_q)$. This character vanishing has been checked for small values of $n, q, j$, but the rapid growth of the character degrees involved has resulted in quite limited data.

In characteristic 0, the analog of Problem~\ref{prob:module-structure} has been solved for Weyl groups $W$ of type ABC \cite{BR, MRW}. In this case, the exterior degree $j$ piece of $SR_G$ is the permutation representation of $W$ on codimension $j$ faces of the Coxeter complex, after a twist by the sign character. This is in contrast with the simpler structure in \eqref{eq:sl-guess} given by copies of the regular representation. Groups strictly containing $SL_n(\FF_q)$ may exhibit more interesting module structure.

\section*{Acknowledgements}

The authors are grateful to Vic Reiner for helpful conversations and, in particular, asking B. Rhoades about the structure of the superspace $GL_n(\FF_q)$-coinvariant ring after a talk at the University of Minnesota in February 2026. B. Rhoades was partially supported by NSF Grant DMS-2246846. A. Wilson was partially supported by AMS-Simons PUI Grant 434651.
The authors used ChatGPT extensively to write code which helped them guess the statements of Lemma~\ref{lem:syzygy} and Proposition~\ref{prop:steinberg}. The proofs of Lemma~\ref{lem:syzygy} and Proposition~\ref{prop:steinberg} were assisted in part by long conversations with ChatGPT. We have also used ChatGPT to assist in revising this paper. The authors wrote this entire manuscript and take full responsibility for its contents.

\maketitle

\end{document}